\documentclass[12pt]{article}
\usepackage{latexsym}
\usepackage{amsmath}
\usepackage{amsthm}
\usepackage{amssymb}
\usepackage{booktabs}
\usepackage{amsfonts}
\usepackage{mathrsfs} 
\usepackage{float}
\usepackage{graphics,graphicx}
\usepackage{xcolor}
\usepackage[colorlinks=true]{hyperref}
\hypersetup{urlcolor=blue, citecolor=red}%
\usepackage{comment}

\numberwithin{equation}{section}

\newtheorem{theorem}{Theorem}[section]

\newtheorem{proposition}{Proposition}[section]
\newtheorem{corollary}[theorem]{Corollary}

\newtheorem{remark}[theorem]{Remark}
\newtheorem{assumption}{Assumption}[section]

\newcommand{\eps}{\varepsilon}

\newcommand{\supp}{{\rm supp}}
\renewcommand{\S}{{\mathbb S}}
\newcommand{\R}{\mathbb R}
\newcommand{\C}{\mathbb C}
\newcommand{\N}{\mathbb N}
\newcommand{\K}{\mathcal K}
\newcommand{\la}{\langle}
\newcommand{\ra}{\rangle}

\newcommand{\rmd}{\mathrm{d}}
\newcommand{\bfi}{\mathbf{i}}

\title{Inverse Born series based neural operators}
\author{John C Schotland\thanks{Zhao and Ji Professor of Mathematics, Yale University, New Haven, CT 06520-8283, USA. Email: john.schotland@yale.edu}\qquad Aseel Titi\thanks{Mathematics Department, Embry-Riddle Aeronautical University, Prescott, Arizona, USA. Email:  titia@erau.edu}
\qquad Jenn-Nan Wang\thanks{Institute of Applied Mathematical Sciences, National Taiwan University, Taipei 106, Taiwan. Email: jnwang@math.ntu.edu.tw}}

\date{}

\begin{document}
\maketitle

\begin{abstract}
The inverse Born series provides a general perturbative framework for representing nonlinear inverse maps between infinite-dimensional function spaces and has found numerous applications in inverse problems governed by partial differential equations and integral equations. Motivated by its operator-theoretic structure, we develop a systematic framework for constructing neural operators that approximate the operator expansions arising in the inverse Born series. Our approach combines the analytical representation of the inverse Born series with the expressive power of neural operators, yielding data-driven approximations of the nonlinear inverse map while preserving the underlying operator structure.

The proposed framework is applicable to a broad class of inverse problems and is presented in a general functional-analytic setting. To demonstrate its practical performance, we consider two representative examples: inverse scattering and the Calder\'on (electrical impedance tomography) problem. Numerical experiments show that the constructed neural operators accurately approximate the inverse Born expansions and produce high-quality reconstructions across a range of test cases. These results indicate that the proposed methodology provides an effective and computationally efficient approach for learning nonlinear inverse operators and suggests a promising direction for integrating classical operator expansions with modern neural operator architectures.
\end{abstract}

\section{Introduction}

The problem considered in this work is to develop a systematic framework for constructing \emph{neural operators} that approximate the operator expansion arising in the inverse Born series. This series provides a powerful and general methodology for representing nonlinear  inverse maps between infinite-dimensional function spaces. The inverse Born series was introduced in the setting of one-dimensional quantum mechanical inverse scattering.
However, the method is not restricted to scattering problems. Rather, it constitutes a general perturbative framework that has been developed and applied to a broad class of forward and inverse problems governed by partial differential equations and integral equations, including those of optical tomography, electrical impedance tomography, and acoustic and electromagnetic imaging \cite{arridge2012inverse, kilgore2012inverse, kilgore2017convergence, machida2015inverse, markel2003inverse, moskow2009numerical, panasyuk2006nonlinear}. We refer to \cite{weglein2003inverse,moskow201912} and the references therein for a comprehensive exposition of the method and its applications.  This work is primarily concerned with the inverse Born series. However, the techniques developed here can be extended in a straightforward manner to  inversion of other multilinear operator expansions.

At a fundamental level, many forward problems  can be expressed as nonlinear operator equations of the form
\[
\mathcal{F}(q) = f,
\]
where $q$ denotes an unknown coefficient, source, or medium parameter, $f$ represents observable data, and $\mathcal{F}$ is a nonlinear operator acting between suitable function spaces. The corresponding inverse problem seeks to recover $q$ from measurements of $f$, a task that is often ill-posed and highly sensitive to noise \cite{isakov2006inverse}. The Born series provides a formal expansion of the forward operator $\mathcal{F}$ around a known reference coefficient, while the inverse Born series provides an expansion of the inverse map $\mathcal{F}^{-1}$ in the data $f$ under appropriate assumptions. The convergence and stability of the Born and inverse Born series have been rigorously analyzed in \cite{moskow2008convergence,hoskins2022analysis}.

Recent advances in machine learning, in particular the development of \emph{neural operators}, offer a promising alternative to explicit evaluation of the Born or inverse Born series. Neural operators are designed to approximate mappings between infinite-dimensional function spaces. Architectures such as PCA-Net, Fourier Neural Operators (FNO) and DeepONet have demonstrated the ability to learn solution operators of PDEs directly from data \cite{anandkumar2020neural, bonev2023spherical, bhattacharya2021model, kovachki2023neural, li2020neural, li2020fourier, li2020multipole, li2024physics, lu2019deeponet}.

The inverse Born series provide a natural analytical foundation for the construction of neural operators. Each term in the series corresponds to a multilinear operator acting on a function space, suggesting a close conceptual connection with the neural operator architecture. Rather than explicitly computing each term in the series, neural operators can be trained to approximate the entire series, effectively learning a resummed representation of the inverse map. This approach preserves the operator-theoretic structure of the problem, while alleviating the computational and convergence limitations of classical series expansions. Moreover, constructing neural operators inspired by the inverse Born series enables a principled integration of physical models and data-driven learning. The series structure provides interpretability and theoretical grounding, while neural operators offer expressive power and flexibility. This hybrid framework is particularly attractive for inverse problems, where learning the inverse Born map directly can lead to stable and efficient reconstructions.  In recent years, machine learning methods have attracted considerable attention as alternative numerical approaches for solving inverse problems; related developments can be found in \cite{adler2017solving, arridge2019solving, schonlieb2025data, bubba2024data}.

Universal approximation theorems for operator learning have been established for several  architectures, including neural operators \cite{kovachki2021universal,kovachki2023neural,lanthaler2023nonlocal,kratsios2024mixture}, DeepONet \cite{lu2019deeponet,lanthaler2022error}, and PCA-Net \cite{lanthaler2023operator}. A comprehensive overview of recent advances in neural operator theory can be found in \cite{kovachki2024operator}. The objective of this work is to establish \emph{quantitative} universal approximation theorems, which provide explicit upper bounds on the number of learnable parameters required to achieve a prescribed level of approximation accuracy. To place our results in context, it has been shown in \cite{lanthaler2023curse} that operator learning for general classes of operators is subject to the so-called "curse of parametric complexity,'' whereby the number of parameters required for approximation grows exponentially as the target accuracy increases. Such unfavorable scaling may severely limit the practical utility of neural operators.

Nevertheless, when attention is restricted to operators arising from specific classes of partial differential equations, recent work has demonstrated that quantitative approximation theorems with tractable parametric complexity can be achieved, avoiding exponential growth. For example, quantitative approximation results have been established for the Darcy and Navier-Stokes equations using Fourier neural operators and PCA-Net, respectively \cite{kovachki2021universal,lanthaler2023operator}. Further developments include the analysis of Hamilton-Jacobi equations via Hamilton-Jacobi neural operators \cite{lanthaler2023curse}, and nonlinear parabolic equations in \cite{furuya2024quantitative}. In addition, neural operators for the scattering problem, mapping the refractive index to the scattering amplitude, were constructed in \cite{scaop2025}, where corresponding quantitative approximation bounds were also derived. In addition, quantitative approximation theorems based on DeepONet architectures have been developed for a broad range of PDEs, including elliptic, parabolic, and hyperbolic equations \cite{chen2023deep,lanthaler2022error,marcati2023exponential}, while convergence and approximation results for advection-diffusion equations were investigated in \cite{deng2021convergence}. Recently, a quantitative universal approximation of neural operators for the Dirichlet-to-Neumann map was proved in \cite{KAFW25}.

%\subsection{Contributions of this work}

In summary, this work is motivated by the observation that the Born and inverse Born series provide general operator expansions for nonlinear forward and inverse problems, extending well beyond classical scattering theory. Since we are primarily interested in inverse problems, we build on this insight to construct neural operators whose architectures are directly informed by the combinatorial structure of the inverse Born series, following recent developments in \cite{furuya2024quantitative,scaop2025,KAFW25}. Rather than employing the full inverse Born series, we construct our neural operators using a \emph{reduced} inverse Born series, which leads to a more efficient and transparent framework. Furthermore, we derive the dependence of the parametric complexity on the approximation accuracy whenever possible.  By using these series as both a conceptual and architectural foundation, the resulting neural operators yield scalable and stable approximations of complex nonlinear operators while preserving the underlying physical and operator-theoretic structure. This perspective further advances the theory of operator learning and broadens its applicability to inverse problems across a wide range of physical settings. 
%As noted above, the techniques developed here readily extend to the construction of neural operators for the Born series. We omit this straightforward generalization for brevity.

This paper is organized as follows. In Section~\ref{sec2}, we review the theoretical foundations of the Born and inverse Born series. Section~\ref{sec3} presents the main architecture of the proposed neural operators in detail. In Section~\ref{sec4}, we apply the framework developed in Section~\ref{sec3} to several representative physical examples. Numerical simulations illustrating the performance of the proposed methods are presented in Section~\ref{sec5}.

\section{Born and inverse Born series}\label{sec2}

In this section, we briefly review the Born and inverse Born series, which provide perturbative operator expansions for nonlinear forward and inverse problems. We present the series in an abstract operator-theoretic setting that highlights their general applicability beyond classical scattering theory and prepares the groundwork for the neural operator constructions introduced in subsequent sections.

Let $X$ and $Y$ be separable Hilbert spaces. We consider the following power series (Born series)
\begin{equation}\label{born}
\phi=K_1(\eta)+K_2(\eta,\eta)+K_3(\eta,\eta,\eta)+\cdots,
\end{equation}
where $K_j:X^j\to Y$. The series \eqref{born} defines the forward mapping ${\mathcal F}:\eta\mapsto\phi$. The precise form of the $K_j$ depends on the problem at hand. We will study the $K_j$ in greater detail later.
%%%%%%%%
%%%%%%%%
\begin{comment}
In many situations, $K_j$ is given by
\begin{equation*}
\begin{aligned}
&K_j(\eta,\cdots,\eta)(x)\\
&=\int_\Omega\cdots\int_\Omega G(x,x_1)G(x_1,x_2)\cdots G(x_{j-1},x_j)\eta(x_1)\eta(x_2)\cdots\eta(x_j)u(x_j)\rmd x_1\rmd x_2\cdots\rmd x_j,
\end{aligned}
\end{equation*}
where $\Omega$ is a bounded domain in $\R^d$ and $G(x,y)$ an appropriate integral kernel. The multilinear operator $K_j$ can be expressed as an iterated operator 
\begin{equation*}
u_{j+1}=P_\eta[u_j]\quad\text{for}\quad j=0,1,2,\cdots,
\end{equation*}
where 
\[
P_\eta[u]=\int_\Omega G(x,y)\eta(y)u(y)\rmd y
\]
and
\[
u_0=\chi_\Omega.
\]
In other words, we have 
\begin{equation}\label{Kj}
K_j(\eta,\cdots,\eta)(x)=P_\eta^{[j]}[u_0],\;\; j=1,2,\cdots.
\end{equation}
For the later discussion in the construction of neural operators, we assume that the kernel $G$ satisfies that for any $\eps\in(0,1)$, there exists $K=K(\eps)\in \N$ such that
\begin{equation}\label{GK}
\|G_K-G\|_{X\times Y}
\end{equation} 
\end{comment}
%%%%%%%%%%%%%%%%%%%

We consider the inverse problem of determining $\eta$ from $\phi$. To this end, we make the ansatz: 
\begin{equation}\label{iborn}
\eta={\mathcal I}[\phi]=\sum_{m=1}^\infty\K_m(\phi).
\end{equation}
The series \eqref{iborn} is referred as the inverse Born series (IBS). To determine the inverse operators $\{\K_m\}_{m=1}^\infty$, we substitute \eqref{iborn} into \eqref{born}, which yields
\begin{align}
\K_1(\phi)&=K_1^{\dagger}(\phi),\label{k1}\\
\K_2(\phi)&=-\K_1(K_2(\K_1(\phi),\K_1(\phi))),\label{k2}\\
\K_3(\phi)&=-\K_1(K_2(\K_1(\phi),\K_2(\phi)))-\K_1(K_2(\K_2(\phi),\K_1(\phi)))\nonumber\\
&\quad\,-\K_1(K_3(\K_1(\phi),\K_1(\phi),\K_1(\phi))),\label{k3}\\
\K_m(\phi)&=-\K_1\sum_{n=2}^{m}K_n\Bigl(\sum_{i_1+\cdots+i_n=m}\K_{i_1}(\phi)\otimes\cdots\otimes\K_{i_n}(\phi)\Bigr),\;\;m\ge 4,\label{k4}
\end{align}
where $K_1^\dagger$ is a regularized pseudoinverse of $K_1$ \cite{hoskins2022analysis,moskow201912}. If $K_1^{-1}$ is bounded, we simply choose $\K_1=K_1^{-1}$. In general, $\K_1$ can be defined as follows. Consider the problem of Tikhonov regularization:
\[
\eta^\dagger=\text{arg}\min_\eta\,(\|K_1\eta-\phi\|_Y^2+\lambda F(\eta)),
\]
where $F$ is a convex penalty function. We then define $\K_1(\phi)=\eta^\dagger$. In particular, if $F(\eta)=\|\eta\|^2_X$, then
\begin{equation*}
\K_1=(K_1^\ast K_1+\lambda I)^{-1}K_1^\ast.
\end{equation*}
In the subsequent section, we will employ a low-rank approximation of $K_1$ and choose $\K_1$ accordingly.

We recall the following convergence theorem and error estimate of the IBS established in \cite{hoskins2022analysis}. 
\begin{theorem}{\rm\cite[Theorem~2.2]{hoskins2022analysis}}\label{thm:IBS}
Let $\mu$ and $\nu$ be positive constants, and suppose that the multilinear operators
$K_m : X^m \to Y$ satisfy
\[
\|K_m(\eta_1,\ldots,\eta_m)\|_Y \le \nu \mu^{m-1}
\|\eta_1\|_X \cdots \|\eta_m\|_X,
\qquad m = 1,2,\ldots .
\]
Then the IBS converges provided that
\[
\|\K_1 \phi\|_X < r,
\]
where the radius of convergence $r$ is given by
\[
1=
\frac{1}{2\mu}[\sqrt{16c^2 + 1} - 4c]\quad\text{with}\quad c= \max\{2, \|\K_1\|\,\nu\}.
\]
Moreover, if $\K_1 \phi \in B_r$, then the inverse operator $\mathcal{I}$ maps $B_r$
into $B_{r_0}$, where
\[
r_0 = \frac{2\mu}{\sqrt{16c^2 + 1}},
\]
where $B_r$ and $B_{r_0}$ are balls of radius $r, r_0$ in $X$, respectively. 
\end{theorem}

\begin{theorem}{\rm\cite[Theorem~2.4]{hoskins2022analysis}}\label{thm:approx_error}
Suppose that the hypotheses of {\rm Theorem~\ref{thm:IBS}} hold and that both the forward Born series and the IBS converge. Let $\tilde{\eta}$ denote the sum of the IBS, and let
\[
\eta_1 = \K_1 \phi.
\]
Setting
\[
{\mathcal M} = \max\{\|\eta\|_X, \|\tilde{\eta}\|_X\},
\]
we further assume that
\begin{equation*}%\label{eq:M_condition}
{\mathcal M} < \frac{1}{\mu}\left(1 - \sqrt{\frac{\nu\|\K_1\|}{1 + \nu \|\K_1\|}}\right).
\end{equation*}
Then the approximation error satisfies the estimate
\begin{equation}\label{error-est}
\begin{aligned}
\left\|\eta - \sum_{m=1}^{N} \K_m(\phi)\right\|_X
\le\;&\widetilde{\mathcal M} \left(\frac{\|\eta_1\|_X}{r}\right)^{N+1}\frac{1}{1 - \|\eta_1\|_X / r}
\\
&\;
+ \left(1 - \frac{\nu \|\K_1\|}{(1 - \mu{\mathcal M})^2}+ \nu \|\K_1\|\right)^{-1}\|(I - \K_1 K_1) \eta\|_X,
\end{aligned}
\end{equation}
where
\[
\widetilde{\mathcal M} = \frac{2\mu}{\sqrt{16c^2 + 1}}.
\]
\end{theorem}
\begin{remark}
   As pointed out in {\rm\cite{hoskins2022analysis}}, the limit of the IBS does not, in general, coincide with $\eta$. Nevertheless, \eqref{error-est} provides an explicit estimate of the approximation error. 
\end{remark}

To facilitate the subsequent discussion, we present a detailed analysis of $K_1$ and construct a low-rank approximation thereof. The operator $\mathcal{K}_1$ is then defined based on this approximation. Assume that $K_1:X\to Y$ is a compact operator. Let $\{\Psi_n\}_{n\in\N}$ and $\{\Phi_m\}_{m\in\N}$  be an orthonormal basis of $X$ and $Y$, respectively. For $N, M\in\N$, let $P_N$ and $P_M$ be orthogonal projections onto $\{\Psi_n\}_{n=1}^N$ and $\{\Phi_m\}_{m=1}^M$, respectively. Then $P_MK_1P_N$ is represented by an $M\times N$ matrix $C:=(C_{mn})_{m,n=1}^{M,N}$ with
\[
C_{mn}=\la K_1(\Psi_n),\Phi_m\ra_{Y}.
\] 
In other words, for any $\eta\in X$ with $\eta=\sum_{n=1}^\infty\la\eta,\Psi_n\ra_X\Psi_n$, we have
\[
P_MK_1P_N\eta=\sum_{m,n=1}^{M,N}C_{mn}\la\eta,\Psi_n\ra_X\Phi_m.
\]
Moreover, by Theorem~\ref{app1}, for any $\eps>0$, there exist $M_0,N_0$ such that
\begin{equation}\label{k1error}
\|P_MK_1P_N-K_1\|_{X\to Y}<\eps
\end{equation}
for all $M\ge M_0$ and $N\ge N_0$. That is, $C$ is an approximation of the operator $K_1$. Under suitable regularity assumptions on $\phi$ and $\eta$, the quantities $M_0$ and $N_0$ defined in \eqref{k1error} can be explicitly estimated in terms of $\varepsilon$, as shown in Section~\ref{sec4}.

Motivated by the finite-dimensional approximation $C$, we define the corresponding finite-dimensional Tikhonov regularized inverse by
\[
C^\dagger
=
(C^\ast C+\lambda I)^{-1}C^\ast,
\]
which serves as an approximation to $\mathcal{K}_1=(K_1^\ast K_1+\lambda I)^{-1}K_1^\ast$. To see this, denote
\[
{\K}^{M,N}_1:=\left((P_MK_1P_N)^\ast(P_MK_1P_N)+\lambda I\right)^{-1}(P_MK_1P_N)^\ast. 
\]
Then
\begin{align*}
&\mathcal{K}_{1}-{\mathcal{K}}^{M,N}_{1}\\
&=
\bigl((K_1^{*}K_1+\lambda I)^{-1}
-\left((P_MK_1P_N)^\ast(P_MK_1P_N)+\lambda I\right)^{-1}\bigr)K_1^{*} \\
&\qquad+\left((P_MK_1P_N)^\ast(P_MK_1P_N)+\lambda I\right)^{-1}
(K_1^{*}-(P_MK_1P_N)^\ast) \\
&=
(K_1^{*}K_1+\lambda I)^{-1}
((P_MK_1P_N)^\ast(P_MK_1P_N)-K_1^{*}K_1)\times\\
&\qquad\quad((P_MK_1P_N)^\ast(P_MK_1P_N)+\lambda I)^{-1}K_1^{*}\\
&\qquad+
((P_MK_1P_N)^\ast(P_MK_1P_N)+\lambda I)^{-1}(K_1^{*}-(P_MK_1P_N)^\ast).
\end{align*}
Since
\[
\|(K_1^{*}K_1+\lambda I)^{-1}\|_X\le\frac{1}{\lambda},
\quad
\|\left((P_MK_1P_N)^\ast(P_MK_1P_N)+\lambda I\right)^{-1}\|_X
\le \frac{1}{\lambda},
\]
it follows that
\begin{equation}\label{KK11}
\begin{aligned}
\|\mathcal{K}_{1}-{\mathcal{K}}^{M,N}_{1}\|_{Y\to X}
&\le
\frac{\|K^\ast_1\|_{Y\to X}}{\lambda^{2}}
\|K_1^{*}K_1-(P_MK_1P_N)^\ast(P_MK_1P_N)\|_X\\
&+
\frac{1}{\lambda}
\|K^\ast_1-(P_MK_1P_N)^\ast\|_{Y\to X}.
\end{aligned}
\end{equation}
By \eqref{KK11}, \eqref{k1error}, we then have
\[
\|\mathcal{K}_{1}-{\mathcal{K}}^{M,N}_{1}\|_{Y\to X}\to 0\quad\text{as}\quad M, N\to\infty.
\]
Thus, for sufficiently large $M$ and $N$, the pseudoinverse $C^\dagger$ provides a valid approximation to $\mathcal{K}_1$.

\begin{comment}
Applying the SVD to $C$ gives 
\[
C=U\Sigma V^\ast,
\]
where $U=[u_1,\cdots,u_M]$ and $V=[v_1,\cdots,v_N]$ are orthogonal matrices and 
\[
\Sigma =
\begin{bmatrix}
\sigma_1 & 0        & \cdots & 0        & \cdots & 0 \\
0        & \sigma_2 & \cdots & 0        & \cdots & 0 \\
\vdots   & \vdots   & \ddots & \vdots   &        & \vdots \\
0        & 0        & \cdots & \sigma_r & \cdots & 0 \\
\vdots   & \vdots   &        & \vdots   & \ddots & \vdots \\
0        & 0        & \cdots & 0        & \cdots & 0
\end{bmatrix}
\in \mathbb{R}^{M \times N}
\]
with $\sigma_1\ge\sigma_2\ge\cdots\ge\sigma_r>0$ and $r\le\min\{M,N\}$. Now the pseudo-inverse of $C$ is given by
\[
C^\dagger=V\Sigma^{-1}U^\ast,
\]
where
\[
\Sigma^{-1}=\begin{bmatrix}
\sigma^{-1}_1 & 0        & \cdots & 0        & \cdots & 0 \\
0        & \sigma^{-1}_2 & \cdots & 0        & \cdots & 0 \\
\vdots   & \vdots   & \ddots & \vdots   &        & \vdots \\
0        & 0        & \cdots & \sigma^{-1}_r & \cdots & 0 \\
\vdots   & \vdots   &        & \vdots   & \ddots & \vdots \\
0        & 0        & \cdots & 0        & \cdots & 0
\end{bmatrix}
\in \mathbb{R}^{N \times M}.
\]
\end{comment}
With the help of $C^\dagger$ (choosing $M=M_0$ and $N=N_0$), we can approximate $\K_1(\phi)$ as follows,
\begin{equation}\label{K1}
\K_1^{M_0,N_0}(\phi)=\sum_{n=1}^{N_0}(C^\dagger\phi_{M_0})_n\Psi_n\in X,
\end{equation}
where $(C^\dagger\phi_{M_0})_n$ is the $n$th element of the vector $C^\dagger\phi_{M_0}$ and\\ $\phi_{M_0}=[\la \phi,\Phi_1\ra_Y,\cdots,\la \phi,\Phi_{M_0}\ra_Y]^\top$. We note that $C^\dagger$ may be regarded as a learnable parameter within the machine learning framework we employ.

In order to construct effective neural operators for the IBS, we would like to find an approximation to the right hand side of \eqref{k4}. We use the reduced inverse Born series to replace the terms on the right-hand side of \eqref{k4}. Let $\eps\in(0,1)$ be given and assume that
\begin{equation}\label{KKE}
\|\K_1 K_1-I\|_X<\eps. 
\end{equation}
Denote by ${\bf K}$ a compact subset of $Y$. In view of \eqref{KKE}, we recall the approximation result in \cite[Lemma~7.1]{ishida2025iterative}, \cite{markel2022reduced}, namely we have that for $m\ge 2$
\begin{equation}\label{cancel1}
\|\K_m(\phi)+\K_1K_2(\K_{m-1}(\phi)\otimes\K_1(\phi))\|_X<C\eps,\quad \;\;\phi\in{\mathbf K}.
\end{equation} 
That is, we can construct $\K_m$ by an approximate iterative scheme:
\begin{equation}\label{ite}
\K_m\simeq -\K_1K_2(\K_{m-1}\otimes\K_1),\quad m\ge 3.
\end{equation}
In view of \eqref{ite}, we define $\widehat{\K}_1=\K_1$, $\widehat{\K}_2=\K_2$,
\begin{equation}\label{iter}
\widehat{\K}_m=-\K_1K_2(\widehat\K_{m-1}\otimes\K_1),\quad m\ge 3,
\end{equation}
and the reduced inverse Born series 
\begin{equation}\label{ribs}
\widehat\eta=\sum_{m=1}^\infty\widehat{\K}_m(\phi).
\end{equation}
The formulas \eqref{iter} and \eqref{ribs} will be used in our construction of neural operators below. 

\begin{remark}
We note that  a rigorous estimate of the reconstruction error $\|\eta-\widehat{\eta}\|$
has not been established. Nevertheless, the reduced inverse Born series serves as a natural approximation to the full inverse Born expansion, and both the theoretical construction and the numerical results suggest that the omitted terms contribute only a limited error under suitable conditions. Establishing rigorous error estimates for the reduced series remains an important direction for future work.
\end{remark}

\begin{comment}
It follows from \eqref{error-est}, \eqref{cancel1}, and the assumption \eqref{KKE} that
\begin{equation}\label{error-est1}
\begin{aligned}
&\left\|\eta -\left(- \sum_{m=1}^{N} \K_1K_2\left(\K_{m-1}(\phi)\otimes\K_1(\phi)\right)\right)\right\|_X\\
&\le\widetilde{\mathcal M} \left(\frac{\|\eta_1\|_X}{r}\right)^{N+1}\frac{1}{1 - \|\eta_1\|_X / r}+ \left(1 - \frac{\nu \|\K_1\|}{(1 - \mu{\mathcal M})^2}+ \nu \|\K_1\|\right)^{-1}\eps\\
&\;\;+CN\eps. 
\end{aligned}
\end{equation}
Assume that 
\[
\frac{\|\eta_1\|_X}{r}\le\delta<1.
\]
Optimizing the right-hand side of \eqref{error-est1} with respect to $N$, we choose
\begin{equation}\label{NN}
N\sim \frac{1}{|\log\delta|}\log\left(\frac{|\log\delta|}{\eps}\right)
\quad\text{when}\quad\eps\ll 1.
\end{equation}
Consequently,
\begin{equation}\label{error-est2}
\left\|\eta -
\left(
-\sum_{m=1}^{N}
\K_1K_2\bigl(\K_{m-1}(\phi)\otimes\K_1(\phi)\bigr)
\right)
\right\|_X
\le
O\!\left(\eps\log\left(\frac{1}{\eps}\right)\right).
\end{equation}
\end{comment}

\section{Architecture of neural operators}\label{sec3}

In this section, we introduce the framework for constructing neural operators based on the reduced inverse Born series \eqref{ribs} and the formulas \eqref{k1}, \eqref{iter}. The input to the neural operators is $\phi\in Y$ and the output is $\eta\in X$. As above, let $\{\Psi_n\}$ and $\{\Phi_m\}$ be orthonormal bases of $X$ and $Y$, respectively.  An $L$-layer  neural operator is constructed by composing layers of the following form: 
\begin{equation}\label{iterative}
\sigma_L(U)=\sigma^{(L)}\circ\sigma^{(L-1)}\circ\cdots\circ\sigma^{(1)}(U),
\end{equation}
where $U\in Y^{d_0}$ is an input function and 
\begin{equation}\label{l-layer}
\sigma^{(\ell)}(U)=\sigma \left( W^{(\ell)}U(x) + ({\mathscr K}_{J}^{(\ell)}U)(x) 
+ b^{(\ell)}\right), \;\; 1 \le \ell \le L .
\end{equation}
Here $d_0=1$, $\sigma : \mathbb{R} \to \mathbb{R}$ is a nonlinear activation function operating component-wise, $W^{(\ell)} \in \mathbb{R}^{d_{\ell}\times d_{\ell-1}}$ is a weight matrix, $b^{(\ell)}$ is a bias, and ${\mathscr K}^{(\ell)}_{J} : Y^{d_{\ell-1}} \to Y^{d_{\ell}}$ is defined by
\begin{equation}\label{KJ}
({\mathscr K}^{(\ell)}_{J}U)(x) :=
\sum_{m,n\in J} C_{nm}^{(\ell)} \langle U, \Psi_{m}\rangle_X
\Psi_{n}(x) \quad \text{with\; $C^{(\ell)}_{nm} \in \mathbb{R}^{d_{\ell} \times d_{\ell-1}}$},
\end{equation}
%and
%\begin{equation*}
%b_{J_2}^{(\ell)}(U)=\sum_{m=1}^{J_2}D_m^{(\ell)}\la U,\Psi_m\ra_X \quad \text{with\; $D^{(\ell)}_{m} \in \mathbb{R}^{d_{\ell} \times d_{\ell-1}}$},
%\end{equation*}
where $ \langle U, \Phi_{m}\rangle_Y$ denotes the inner product operating component-wise. A more precise formulation of ${\mathscr K}^{(\ell)}_{J}$ will be given subsequently. 

The architecture of the neural operators is described as follows. For a given $\eps>0$, let $M_0, N_0\in\N$ be chosen such that the error estimate \eqref{k1error} holds. 

\medskip\noindent
{\bf 1. (Input layer)}. For any $\phi\in Y$, we set 
$$\hat u_1(x)=\K_1^{M_0,N_0}(\phi):=\sum_{n=1}^{N_0}\left(\sum_{m=1}^{M_0}C_{mn}^{(0)}\la\phi,\Phi_m\ra_Y+b^{(0)}\right)\Psi_n(x).$$
This layer is sometimes referred to as an \emph{encoder}.

\medskip\noindent{\bf 2. (Hidden layers)}. Choose $M_1\in\N$.

\vspace{2pt}
For $r=2:M_1$

\vspace{3pt}\hspace{5pt} $U_{r-1}=[\hat u_1,\hat u_{r-1}]^\top$;

\vspace{3pt}\hspace{5pt} $\hat{u}_r = \K_1^{M_0,N_0}(I[A^{(r)}\sigma_{L_r}(U_{r-1})]),\;\; A^{(r)}\in\R^{1\times d_{L_{r}}}$,\;\; $I:X\to Y$\;\;a linear map; 

\vspace{2pt} End

\medskip\noindent{\bf 3. (Output layer)}. $\eta_{\text{app}}=\sum_{r=1}^{M_1}\hat u_r$.

\medskip\noindent
Here,  $C^{(0)}_{mn}, b^{(0)}, W^{(\ell)}, C_{nm}^{(\ell)}, A^{(r)}$ are learnable parameters and $M_0, N_0, M_1$, ${\bf d}=(d_1,\cdots,d_{L_m})$, ${\bf L}=(L_1,\cdots,L_{r})$ are appropriately chosen indices.
\begin{comment}
The central theme of this paper is to establish universal approximation theorems for neural operators together with explicit and quantitative estimates of their parametric complexity. In particular, we derive bounds that characterize how the approximation accuracy influences the number of parameters required to achieve a prescribed error. While some of our parametric complexity estimates depend on the approximation accuracy only implicitly, they nevertheless provide rigorous quantitative guarantees on the expressive power and efficiency of the proposed neural operator architectures.
\end{comment}

\section{Examples}\label{sec4}

In this section, we apply the neural operator framework developed above to the inverse scattering problem and the Calder\'on problem. Through these applications, we demonstrate how the proposed methodology can be instantiated in concrete problem settings, highlighting both its flexibility and effectiveness in modeling operator mapping.

\subsection{Inverse scattering problem}

We consider the following scattering problem for an acoustic wave in an inhomogeneous medium:
\begin{equation}\label{scatter}
\Delta u+k^2(1+\eta(x))u=0\quad\text{in}\quad\R^d,\;\; d=2,3 ,
\end{equation}
where $\supp(\eta)\subset\Omega$ and $\Omega$ is a  compact subset of $\R^d$. 
Let $u_\eta:=u^{\rm inc}+u_\eta^{\rm sca}$ satisfy  \eqref{scatter} along with the Sommerfeld radiation condition
\begin{equation}\label{sommerfeld}
\lim_{|x|\to\infty}|x|^{\frac{d-1}{2}}\left(\frac{\partial u_\eta^{\rm sca}}{\partial |x|}-\bfi ku_\eta^{\rm sca}\right)=0.
\end{equation}
If follows that $u^{\rm inc}$ solves the Helmholtz equation in $\R^d$.
%Assume that $u^{\rm inc}$ is the plane incident field, i.e. $u^{\rm inc}=e^{\bfi k\cdot\theta}$ with $\theta\in\S^{2}$. 
It is well-known that the scattered field $u_\eta^{\rm sca}$ possesses the asymptotic behavior
\begin{equation}\label{sca}
u_\eta^{\rm sca}(x)=\frac{e^{\bfi k|x| }}{|x|^{\frac{d-1}{2}}}\left(u_\eta^\infty(\hat{x})+O\left(\frac{1}{|x|}\right)\right)\quad\mbox{\rm as}\quad|x|\to\infty,
\end{equation}
where $\hat x=x/|x|\in\S^{d-1}$ and $u_\eta^\infty(\hat{x})$ is called the scattering amplitude. 
The scattering amplitude $u_\eta^\infty(\hat{x})$ is given  by
\begin{equation}\label{far-field-formula}
u_\eta^\infty(\hat{x})= c_d \int_{\Omega}e^{-\bfi k\hat{x} \cdot y} \eta(y) u_\eta(y) \, \rmd y ,
\end{equation}
where
\[
c_d=\left\{
\begin{aligned}
&e^{\bfi\frac{\pi}{4}}\sqrt{\frac{k^3}{8\pi}}&\quad\mbox{for}\quad d=2,\\
&\frac{k^2}{4\pi}&\quad\mbox{for}\quad d=3,
\end{aligned}\right.
\]
It follows that the total field $u_\eta$ obeys the {Lippmann-Schwinger equation}
\begin{equation}
\label{LS-eq}
u_\eta(x)=u^{\rm inc}(x)+k^2\int_{\Omega}\Phi(x,y)\eta(y)u_\eta(y) \rmd y,\;\; \; x\in\R^d,
\end{equation}
where $G(x,y)$ denotes the outgoing fundamental solution of the Helmholtz equation in $\mathbb{R}^d$:
\begin{equation}
\label{fundamental-sol}
G(x,y):=\left\{
\begin{aligned}
&\frac{\bfi}{4}H_0^{(1)}(k|x-y|)\quad\mbox{if}\quad d=2,\\
&\frac{1}{4\pi}\frac{e^{\bfi k|x-y|}}{|x-y|}\quad\mbox{if}\quad d=3,
\end{aligned}\right.
\end{equation} 
where $H_0^{(1)}(z)$ is the Hankel function of the first kind of order zero.
%%%%%
%%%%%

Let us denote 
\begin{equation*}
L^\infty_M(\Omega; \C)=\{ \eta \in L^\infty(\R^d; \C) :  0 \le{\rm Im}(\eta), \ \supp\eta \subset \Omega,\ |\eta(x)|\le M\ \mbox{a.e.} \}. 
\end{equation*}
We first study the Born series. Denote the mapping
\begin{equation}\label{fix-point-eq}
u = \Psi_{\eta}(u), 
\end{equation}
where $\Psi_{\eta} : L^2(\Omega; \C) \to L^2(\Omega; \C)$ is defined by 
\begin{equation}\label{K-op}
\Psi_{\eta}(u)(x) := u^{\rm inc}(x) + k^2 \int_\Omega G(x,y)\eta(y)u(y) \rmd y, \;\;\forall\; x \in \Omega.
\end{equation}
The above map between $L^2(\Omega; \C)$ is well-defined, since $G \in L^2(\Omega \times \Omega; \C)$. We will consider the case where $\Psi_\eta$ is a contraction mapping. Therefore, we impose

%By \cite[Theorem 8.2]{CK19scattering}, we can see that $K_{k,q} : L^2(D; \C) \to L^2(D; \C)$ is compact. 

\begin{assumption}\label{ass:small-k}
We assume that the wavenumber $k>0$ and the upper bound $M>0$ satisfies
\[
k^2 M \| G \|_{L^2(\Omega\times \Omega)} \le \rho<1.
\]
\end{assumption}
\noindent It is not difficult to prove the following results.
\begin{proposition}
Under {\rm Assumption~\ref{ass:small-k}}, $\Psi_{\eta} : L^2(\Omega;\C) \to L^2(\Omega;\C)$ satisfies
\[
\|\Psi_\eta(u)-\Psi_\eta(v)\|_{L^2(\Omega)}\le\rho\|u-v\|_{L^2(\Omega)}
\] 
for all $u, v\in L^2(\Omega;\C)$. This implies that $u_\eta$ is a fixed point of $\Psi_\eta$. 
\end{proposition}
\begin{corollary}\label{cor1}
Suppose that {\rm Assumption~\ref{ass:small-k}} holds and $u^{\rm inc}$ satisfies 
\[
Q:=k^2\Big{\|}\int_\Omega G(x,y)u^{\rm inc}(y)\eta(y)\rmd y\Big{\|}_{L^2(\Omega)}<\infty.
\]
Assume that 
\begin{equation*}\label{R0}
\frac{Q}{1-\rho}\le \beta
\end{equation*}
and let $B(u^{\rm inc},\beta):=\{v\in L^2(\Omega;\C): \|v-u^{\rm inc}\|_{L^2(\Omega)}\le \beta\}$. Then there exists a unique fixed point $u\in B(u^{\rm inc},\beta)$ of \eqref{K-op}. 
\end{corollary}
\begin{remark}\label{rem1}
For simplicity, we consider a weaker result. Let us set
\[
Q':=k^2M\|G\|_{L^2(\Omega\times \Omega)}\|u^{\rm inc}\|_{L^2(\Omega)}
\]
and choose $S'$ satisfying
\[
\frac{Q'}{1-\rho}\le \beta'.
\]
Now if we take 
\begin{equation}\label{R}
R=\beta'+\|u^{\rm inc}\|_{L^2(\Omega)}, 
\end{equation}
then \eqref{K-op} has a unique fixed point in $B(0,R)$. 
\end{remark}
Therefore, the total field $u_\eta$ to \eqref{scatter} is now expressed as
\begin{equation}\label{scatter-Born}
\begin{aligned}
u_\eta(x)&=u^{\rm inc}(x)+\int_\Omega k^2G(x,y)\eta(y) u^{\rm inc}(y)\rmd y\\
&\quad+\int_\Omega k^2G(x,y)\eta(y)\left(\int_\Omega k^2G(y,z)\eta(z)u^{\rm inc}(z)\rmd z\right)\rmd y+\cdots\\
&=u^{\rm inc}(x)+\sum_{j=1}^\infty [T\eta]^{(j)}(u^{\rm inc})(x),
\end{aligned}
\end{equation}
where
\[
\begin{aligned}
&\quad([T\eta]^{(j)}f)(x)\\
&=\int_\Omega\cdots\int_\Omega (k^2G)(x,x_1)(k^2G)(x_1,x_2)\cdots (k^2G)(x_{j-1},x_j)\eta(x_1)\eta(x_2)\cdots\eta(x_j)f(x_j)\rmd x_1\rmd x_2\cdots\rmd x_j,
\end{aligned}
\]
If we take $u^{\rm inc}(x)=e^{\bfi k\theta\cdot x}$ with $\theta\in\S^{d-1}$, then the scattering amplitude is given by
\begin{equation*}
\begin{aligned}
u_\eta^\infty(\theta,\hat{x})&= c_d \int_{\Omega}e^{-\bfi k(\hat{x}-\theta)\cdot y} \eta(y) \rmd y+c_d\sum_{j=1}^\infty \int_\Omega e^{-\bfi k\hat{x}\cdot y} \eta(y) [T\eta]^{(j)}(e^{\bfi k\theta\cdot\bullet})(y)\rmd y\\
&=K_1(\eta)(\theta,\hat{x})+K_2(\eta\otimes\eta)(\theta,\hat{x})+K_3(\eta\otimes\eta\otimes\eta)(\theta,\hat{x})+\cdots,
\end{aligned}
\end{equation*}
where 
\[
\left\{\begin{aligned}
&K_1(\eta)(\theta,\hat{x})=c_d\int_\Omega e^{-\bfi k(\hat{x}-\theta)\cdot y} \eta(y) \rmd y,\\
&K_2(\eta\otimes\eta)(\theta,\hat{x})=c_d\int_\Omega e^{-\bfi k\hat{x}\cdot y}\eta(y)\left(k^2\int_\Omega G(y,z)\eta(z)e^{\bfi k\theta\cdot z}\rmd z\right)\rmd y,\\
&K_3(\eta\otimes\eta\otimes\eta)(\theta,\hat{x})=c_d\int_\Omega e^{-\bfi k\hat{x}\cdot y}\eta(y)\left(k^2\int_\Omega G(y,z)\eta(z)\left(k^2\int_\Omega G(z,w)\eta(w)e^{\bfi k\theta\cdot w}\rmd w\right)\rmd z\right)\rmd y,\\
&\cdots
\end{aligned}\right.
\]

Let the orthonormal basis $\{\Psi_n\}_{n\in\N}$ be given by the Dirichlet eigenfunctions of $-\Delta$ on $\Omega$ with associated eigenvalues $\{\lambda_n\}_{n\in\N}$, so that $-\Delta \Psi_n=\lambda_n\Psi_n$ and $\Psi_n\in H^1_0(\Omega)$. We also denote by $\{\Phi_\ell\}_{\ell\in\N}$ the eigenfunctions of the Laplace-Beltrami operator on $\S^{d-1}$. Let
\[
\begin{aligned}
A^{(1)}_{\ell m n}&=c_d\iiint_{\S^{d-1}\times\S^{d-1}\times\Omega} e^{-\bfi k(\hat{x}-\theta)\cdot y}\Phi_\ell(\hat x)\Phi_m(\theta)\Psi_n(y)\rmd y\rmd\sigma_{\hat x}\rmd\sigma_\theta.
\end{aligned}
\]
Then $K_1(\eta)(\theta,\hat x)$ can be written as 
\[
K_1(\eta)(\theta,\hat x)=\sum_{\ell, m=1}^\infty\left( \sum_{n=1}^\infty A^{(1)}_{\ell mn}\la\eta,\Psi_n\ra_X\right)\Phi_\ell(\hat x)\Phi_m(\theta),
\]
where the series converges at least in $L^2(\Omega)$ uniformly in $\theta,\hat x\in\S^{d-1}$. 
Furthermore, in view of Weyl's asymptotic formula $\lambda_n\simeq n^{2/d}$, for $\eps\in(0,1)$ and $\eta\in H^s(\Omega)$, we choose 
\begin{equation}\label{Nn1}
N_1=O(\eps^{-s/d}), 
\end{equation}
then 
\begin{equation}\label{N1}
\Bigl{\|}\eta-\sum_{n=1}^{N_1}\la\eta,\Psi_n\ra_X\Psi_n\Bigl{\|}_{L^2(\Omega)}<\eps.
\end{equation}
On the other hand, by the analyticity of $K_1(\eta)(\theta,\hat x)$ in $(\theta,\hat x)$, we have that if 
\begin{equation}\label{Nn2}
N_2=O(\log(1/\eps)),
\end{equation}
and using \eqref{N1}, we have
\begin{equation}\label{L1}
\left\|K_1(\eta)-\sum_{\ell, m=1}^{N_2}\left( \sum_{n=1}^{N_1}A^{(1)}_{\ell mn}\la\eta,\Psi_n\ra_X\right)\Phi_\ell\otimes\Phi_m\right\|_{Y}<C\eps
\end{equation}
for all $\eta$ satisfying $\|\eta\|_{H^s(\Omega)}\le M$ with $C=C(M)$.

When $d=2,3$, the kernel $G(x,y)\in H^s(\Omega\times\Omega)$ with $s<2-\frac d2$ \cite{scaop2025}. Therefore, we can write
\begin{equation*}%\label{green}
G(x,y)=\sum_{m,n=1}^\infty C_{mn}\Psi_m(x)\Psi_n(y),
\end{equation*}
where $C_{mn}=\la G,\Psi_m\otimes\Psi_n\ra_{L^2(\Omega\times\Omega)}$ and the series converges in $L^2(\Omega\times\Omega)$. In other words, we can write
\[
[T\eta](f)(x)=\sum_{m,n=1}^\infty (k^2C_{mn}\la \eta f, \Psi_n\ra_X)\Psi_m(x).
\]
Let
\[
[S\eta](g)(\hat x)=\int_\Omega e^{-\bfi k\hat x\cdot y}\eta(y)g(y)\rmd y,
\]
then using the operators $T$ and $S$, we can generalize the definitions of $K_j$'s as follows
\[
K_j(\eta_1,\eta_2,\cdots,\eta_j)(\theta,\hat x)=[S\eta_j]\circ[T\eta_{j-1}]\circ\cdots\circ[T\eta_1](e^{\bfi k\theta\cdot\bullet})(\hat x),\;\; j=2,3,\ldots.
\]
In particular, for $K_2$, we have
\begin{equation}\label{K2}
\begin{aligned}
K_2(\eta_1,\eta_2)(\theta,\hat{x})&=[S\eta_2]\circ[T\eta_1](e^{\bfi k\theta\cdot\bullet})(\hat x)\\
&=\int_\Omega e^{-\bfi k\hat{x}\cdot y}\sum_{m,n=1}^\infty (c_dk^2C_{mn}\la \eta_1 e^{\bfi k\theta\cdot \bullet}, \Psi_n\ra_X)\eta_2(y)\Psi_m(y)\rmd y\\
&=\sum_{m,n=1}^\infty c_dk^2C_{mn}\la \eta_1 e^{\bfi k\theta\cdot \bullet}, \Psi_n\ra_X\la \eta_2 e^{-\bfi k\theta\cdot \bullet}, \Psi_m\ra_X\\
&=\sum_{m,n=1}^\infty \widetilde C_{mn}\la \eta_1 e^{\bfi k\theta\cdot \bullet}, \Psi_n\ra_X\la \eta_2 e^{-\bfi k\hat x\cdot \bullet}, \Psi_m\ra_X,
\end{aligned}
\end{equation}
where $\widetilde C_{mn}=c_dk^2C_{mn}$.

To obtain an approximation of $G(x,y)$ with an explicit estimate on the rank, in the case of $d=2,3$, following the computation in \cite{scaop2025}, and choosing
\begin{equation}\label{Nn3}
N_3=O(\eps^{-s/d})
\end{equation}
we then have
\begin{equation}\label{G}
\Big{\|}G-\sum_{m,n=1}^{N_3} C_{mn}\Psi_m\otimes\Psi_n\Big{\|}_{L^2(\Omega\times\Omega)}<\eps. 
\end{equation}
With the help of \eqref{G}, for $N_3$ given above, one has that for all $\eta_1, \eta_2$ with $\|\eta_1\|_{L^\infty(\Omega)}\le M$,  $\|\eta_2\|_{L^\infty(\Omega)}\le M$,
\begin{equation}\label{Ge}
\left|K_2(\eta_1,\eta_2)(\theta,\hat{x})-\sum_{m,n=1}^{N_3} \widetilde C_{mn}\la \eta_1 e^{\bfi k\theta\cdot \bullet}, \Psi_n\ra_X\la \eta_2 e^{-\bfi k\hat x\cdot \bullet}, \Psi_m\ra_X\right|<\eps
\end{equation}
uniformly in $\theta,\hat x\in\S^{d-1}$

\subsubsection{Neural operators architecture}

Here we aim to construct neural operators based on the IBS framework for the inverse scattering problem. The input data is a given scattering amplitude $u_\eta^\infty(\theta,\hat x)=:\phi(\theta,\hat x)$ and the output is the corresponding function $\eta$. Let $Y=L^2(\S^{d-1}\times\S^{d-1})$.

\medskip\noindent
{\bf 1. (Input layer)}. For any $\phi\in Y$, we set 
$$
\hat u_1(x)=\K_1^{N_1,N_2}(\phi):=\sum_{n=1}^{N_1} \left(\sum_{\ell, m=1}^{N_2}C_{\ell m n}^{(0)}\la\phi,\Phi_\ell\otimes\Phi_m\ra_{Y}+b^{(0)}\right)\Psi_n(x).
$$ 

\medskip\noindent{\bf 2. (Hidden layers)}. Choose $M_1\in\N$.

\vspace{2pt}
For $r=2:M_1$

\vspace{3pt}\hspace{5pt} $U_{r-1}=[\hat u_1(y),e^{\bfi k\theta\cdot y},\hat u_{r-1}(z),e^{-\bfi k\hat x\cdot z}]^\top$;

\vspace{3pt}\hspace{5pt} $\hat{u}_r = \K_1^{N_1,N_2}(I[A^{(r)}\sigma_{L_r}(U_{r-1})]),\;\; A^{(r)}\in\R^{1\times d_{L_{M_1}}}$, where $I:f(y,\theta,\hat x)\to\int_\Omega f(y,\theta,\hat x)dy$;

\vspace{2pt} End

\medskip\noindent{\bf 3. (Output layer)}. $\eta_{\text{app}}=\sum_{r=1}^{M_1}\hat u_r$.

\medskip In order to achieve the desired approximation accuracy, the truncation indices $N_1$ and $N_2$ are chosen to satisfy \eqref{Nn1} and \eqref{Nn2}, respectively. Furthermore, the rank $J$ of the hidden layer in \eqref{KJ} is selected according to \eqref{Nn3}. 

\medskip
To understand how a neural network approximates $K_2$ using \eqref{Ge}, we describe the hidden layer in detail. We will use the ReQU activation function, which is defined by
\[
\sigma_2(t):=\left\{
\begin{aligned}
&t^2,\quad t\ge 0,\\
&0,\quad t<0.
\end{aligned}\right.
\]
Recall from \cite[Lemma~2.1]{li2019better}, for any $t,s\in\R$, we have that
\begin{equation}\label{q1}
t=\beta_1^T\sigma_2(\omega_1 t+\alpha_1)
\end{equation}
and
\begin{equation}\label{q2}
ts=\beta_1^\top\sigma_2(\omega_1t+\alpha_1s),
\end{equation}
where
\[
\beta_1=\frac 14\begin{pmatrix}1\\1\\-1\\-1\end{pmatrix},\quad\omega_1=\begin{pmatrix}1\\-1\\1\\-1\end{pmatrix},\quad\alpha_1=\begin{pmatrix}1\\-1\\-1\\1\end{pmatrix}.
\]
That is, one can represent the multiplication of $t$ and $s$ exactly by a hidden layer and an output layer with inner weight matrix $W^{(1)}=\begin{pmatrix}\omega_1&\alpha_1\end{pmatrix}$ and output weight matrix $W^{(2)}=\beta_1^\top$. Using $\sigma_2$ in the layer \eqref{l-layer}, we can represent the sum
\[
\sum_{m,n=1}^{N_3} \widetilde C_{mn}\la \eta_1 e^{\bfi k\theta\cdot \bullet}, \Psi_n\ra_X\la \eta_2 e^{-\bfi k\hat x\cdot \bullet}, \Psi_m\ra_X
\]
by $I[\tilde A\sigma_L(\eta_2(y),e^{\bfi k\theta\cdot y},\eta_1(z),e^{-\bfi k\hat x\cdot z})]$ with a fixed small number of layers. 

\subsection{The Calder\'on problem}

The theory of IBS for the Calder\'on problem was thoroughly developed in \cite{arridge2012inverse}. Here, we recast the problem within the neural operator framework introduced in Section~\ref{sec3}. For comparison, we mention the recent work \cite{de2025extension}, which considers Fourier Neural Operators (FNOs) on the extension of the inverse map (Neumann-to-Dirichlet map), rather than inverse Born series. The authors develop a rigorous framework for solving the Calder\'on inverse conductivity problem in the presence of noisy measurement perturbations.

To facilitate the discussion, we briefly review the problem considered in \cite{arridge2012inverse}. Let $\Omega$ be a bounded domain in $\mathbb{R}^d$ with smooth boundary $\partial\Omega$, where $d \ge 2$. We consider a scalar field $u$ satisfying
\begin{equation}
\nabla \cdot \bigl(\sigma(x)\nabla u\bigr) = 0,
\quad x\in\Omega,
\label{eq:pde}
\end{equation}
where the coefficient (conductivity) $\sigma(x) > 0$ for all $x \in \Omega$. The field $u$ is also assumed to satisfy the Robin boundary condition
\begin{equation}
u + z\sigma \frac{\partial u}{\partial \nu} = g,
\quad x \in \partial\Omega,
\label{eq:robin}
\end{equation}
where $z \ge 0$ is a constant surface impedance, $\nu$ denotes the unit outer normal to $\partial\Omega$, and $g$ is a prescribed current density. A typical choice of $g$ is a unit-strength dipole source,
\begin{equation}
g = \delta_{x_1} - \delta_{x_2},
\quad x_1,x_2 \in \partial\Omega,
\label{eq:dipole}
\end{equation}
where $\delta_{x_1}$ and $\delta_{x_2}$ denote Dirac delta distributions supported at $x_1$ and $x_2$, respectively.

We consider a conductivity of the form
\[
\sigma(x)=\sigma_0\bigl(1+\eta(x)\bigr),
\]
where $\sigma_0=\sigma|_{\partial\Omega}>0$ is constant on $\partial\Omega$. The perturbation $\eta$ is assumed to belong to $L^\infty(B_a)$ and is supported in $B_a \subset \Omega$. Then equation \eqref{eq:pde} is reduced to
\begin{equation}\label{eq:rpde}
-\Delta u=\nabla\cdot(\eta(x)\nabla u)\quad \text{in}\quad\Omega. 
\end{equation}
The solution $u$ to \eqref{eq:rpde} and \eqref{eq:robin} can be written as
\begin{equation}\label{inteq0}
u(x)=u_0(x)+\int_\Omega G(x,y)\nabla\cdot(\eta(y)\nabla u(y))dy,
\end{equation}
where $G$ is the Green function of $-\Delta$ satisfying the boundary condition \eqref{eq:robin} with zero right-hand side. The leading term $u_0(x)$ satisfies $-\Delta u_0=0$ in $\Omega$ and the boundary condition \eqref{eq:robin}. We can see that $u_0(x)$ is expressed by
\begin{equation}\label{u0}
u_0(x)=\frac{1}{z\sigma_0}\int_{\partial\Omega} G(x,y)g(y)dy.
\end{equation}
Integrating  \eqref{inteq0} by parts yields
\begin{equation*}
    u(x)=u_0(x)-\int_\Omega\nabla_yG(x,y)\cdot\nabla u(y)\eta(y)dy.
\end{equation*}

Let $\phi(x)=u_0(x)-u(x)$, then the Born series is  given by
\begin{equation}\label{cborn}
\phi=K_1(\eta)+K_2(\eta,\eta)+\cdots,
\end{equation}
where for $n\ge 1$
\[
\begin{aligned}
(K_n \eta)(x)
&=
(-1)^n
\int_\Omega\eta(y_1)\,\nabla_{y_1} G(y_1,x)
\cdot \nabla_{y_1}\int_\Omega\eta(y_2)\,\nabla_{y_2} G(y_2,y_1)\\
&\quad\cdots
\nabla_{y_{n-1}}\int_\Omega\eta(y_n)\,\nabla_{y_n} G(y_n,y_{n-1})\cdot \nabla_{y_n} u_0(y_n)\,
dy_1 \cdots dy_n.
\end{aligned}
\]
In particular, 
\begin{equation}\label{K_1}
K_1(\eta)=-\int_\Omega\eta(y_1)\nabla_{y_1}G(y_1,x)\cdot\nabla u_0(y_1)dy_1
\end{equation}
and
\begin{equation}\label{K_2}
K_2(\eta,\eta)=\int_\Omega\eta(y_1)\,\nabla_{y_1} G(y_1,x)
\cdot \nabla_{y_1}\int_\Omega\eta(y_2)\,\nabla_{y_2} G(y_2,y_1)
\cdot \nabla u_0(y_2)\,dy_1\,dy_2.
\end{equation}

As in \cite{arridge2012inverse}, let
\[
(Sf)(x)=\int_\Omega\nabla_y G(x,y)\cdot f(y)dy
\]
and
\[
(Tf)(x)=\nabla(Sf)(x) .
\]
Then $S:[L^2(\Omega)]^d\to H^1(\Omega)$ and $T:[L^2(\Omega)]^d\to [L^2(\Omega)]^d$ are bounded operators. Consequently, $S:[L^2(\Omega)]^d\to [L^2(\Omega)]^d$ is a linear compact operator. 

To express $K_1$ in terms of $S$, let $\psi$ be a smooth cutoff function supported in a proper subset $\Omega_0$ satisfying $\overline{B_a}\subset\Omega_0\subset\overline{\Omega_0}\subset \Omega$ and $\psi=1$ on $B_a$. Recall that $\eta$ is supported in $B_a$. Therefore, we have
\[
\begin{aligned}
K_1(\eta)(x)&=-\int_\Omega\eta(y_1)\nabla_{y_1}G(y_1,x)\cdot\nabla u_0(y_1)dy_1\\
&=-\int_\Omega\eta(y_1)\nabla_{y_1}G(y_1,x)\cdot\psi(y_1)\nabla u_0(y_1)dy_1\\
&=-S(\psi\nabla u_0\,\eta).
\end{aligned}
\]
Since $\psi\nabla u_0$ is smooth in $\Omega$, $K_1:[L^2(\Omega)]^d\to [L^2(\Omega)]^d$ is a linear compact operator. 

Under the assumption that $\|\eta\|_{L^\infty(B_a)}<1$, the convergence of the Born series \eqref{cborn} in $L^\infty(\partial\Omega)$ was established in \cite[Proposition~2.1]{arridge2012inverse}. Here, we are interested in the inverse problem of recovering $\eta$ from measurements of $\phi$ on $\partial\Omega$. Following the setup in \cite{arridge2012inverse}, for the dipole source \eqref{eq:dipole}, fixing $x_1 \in \partial\Omega$ and varying $x_2 \in \partial\Omega$ yields a field $\phi$ that depends both on the source location $x_2$ and on the observation point $x \in \partial\Omega$. We denote this dependence explicitly and consequently assume that $\phi \in L^\infty(\partial\Omega \times \partial\Omega)$. To indicate the dependence of $u_0$ on $x_2$, we write $\nabla u_0(y_1)$ as $\nabla_y u_0(y_1,x_2)$ and thus modify
\[
K_1(\eta)(x,x_2)=-\int_\Omega\eta(y_1)\nabla_{y_1}G(y_1,x)\cdot\psi(y_1)\nabla_{y_1} u_0(y_1,x_2)dy_1.
\]

We now apply the neural operator framework developed in Section~\ref{sec2} to the Calder\'on problem. Let $X=L^2(\Omega)$ and $Y=L^2(\partial\Omega\times\partial\Omega)$. Recall that $\{\Psi_n\}_{n\in\mathbb{N}}$ denotes the orthonormal basis of $L^2(\Omega)$ consisting of the Dirichlet eigenfunctions of $-\Delta$ on $\Omega$, with corresponding eigenvalues $\{\lambda_n\}_{n\in\mathbb{N}}$. On the other hand, let $\{\Phi_m\}_{m\in\N}$ be normalized eigenfunctions of the Laplace-Beltrami operator on $\partial\Omega$. 
We further assume that $\partial\Omega$ is analytic. Consequently, we see that $K_1(\eta)(x,x_2)$ is analytic for $x,x_2\in\partial\Omega$. Therefore, if $\eta\in H^s(B_a)$, then similar to estimates \eqref{N1}, \eqref{L1}, we find that there exist 
\begin{equation}\label{NN1}
N_1=O(\eps^{-s/d})\quad\text{and}\quad N_2=O(\log(1/\eps))
\end{equation}
such that
\begin{equation}\label{KK1}
\left\|K_1(\eta)-\sum_{\ell, m=1}^{N_2}\left( \sum_{n=1}^{N_1}A^{(1)}_{\ell mn}\la\eta,\Psi_n\ra_X\right)\Phi_\ell\otimes\Phi_m\right\|_{Y}<C\eps,
\end{equation}
where
\[
A^{(1)}_{\ell mn}=-\iiint_{\partial\Omega\times\partial\Omega\times\Omega}\Psi_n(y_1)\nabla_{y_1}G(y_1,x)\cdot\psi(y_1)\nabla_{y_1} u_0(y_1,x_2)\Phi_\ell(x)\Phi_m(x_2)dy_1d\sigma_x d\sigma_{x_2}.
\]
We can regard $(A_{\ell mn}^{(1)})$ as a $(\ell m)\times n$ matrix and denote its pseudoinverse by ${\mathcal K}_1^{N_1,N_2}$. 

Now we need the approximation to $K_2$ given in \eqref{K_2}. Similarly, we rewrite 
\begin{equation*}%\label{K22}
\begin{aligned}
&K_2(\eta_1,\eta_2)(x,x_2)\\
&=\int_\Omega\eta_1(y_1)\,\nabla_{y_1} G(y_1,x)
\cdot \psi(y_1)\nabla_{y_1}\left(\int_\Omega\eta_2(y_2)\,\nabla_{y_2} G(y_2,y_1)
\cdot \psi(y_2)\nabla u_0(y_2,x_2)\,dy_2\right)dy_1
\end{aligned}
\end{equation*}
for $x,x_2\in\partial\Omega$. The inner integral is related to the operator $T$ defined above since $\psi(y_2)\nabla u_0(y_2,x_2)$ is smooth for $y_2\in\Omega$ and $x_2\in\partial\Omega$. We observe that $K_2(\eta_1,\eta_2)$ can be expressed as a composite operator. Namely, 
\begin{equation}\label{KC}
K_2(\eta_1,\eta_2)=S\big{(}\psi\eta_1 \,T\left((\psi\nabla u_0)\eta_2\right)\big{)}.
\end{equation}
Recall that $T:[L^2(\Omega)]^d\to [L^2(\Omega)]^d$ is a bounded operator. Therefore, for $n,m\in\N$, let
\[
C_{nm}=\int_\Omega \Psi_m(x)(T\Psi_n)(x)dx
\]
then, for any $\eps>0$, there exist $N_3=N_3(\eps)$, $N_4=N_4(\eps)$ such that
\[
\Big{\|}Tf-\sum_{n=1}^{N_3}\sum_{m=1}^{N_4}C_{nm}\langle f,\Psi_n\rangle\Psi_n\Big{\|}_{L^2(\Omega)}<\eps.
\]
In the absence of additional regularity for \(T\), the above argument yields only the existence of \(N_3(\varepsilon)\) and $N_4(\eps)$, without providing an explicit rate of growth as \(\varepsilon\to 0\). In view of \eqref{KC}, $K_2(\eta_1,\eta_2)(x,x_2)$ for $x,x_2\in\partial\Omega$ can be approximated by the following two steps:
\begin{itemize}
\item Input $\eta_2$, compute
\begin{equation}\label{COM1}
f(y_1,x_2)=\sum_{n=1}^{N_3}\sum_{m=1}^{N_4}\sum_{p=1}^{N_5}C_{nmp}\la\eta_2,\Psi_n\ra\Psi_m(y_1)\Phi_p(x_2). 
\end{equation}
\item Input $\eta_1$, compute
\begin{equation}\label{COM2}
\sum_{\ell=1}^{N_6}\sum_{q=1}^{N_7}\widetilde{C}_{\ell q}\la\eta_1(\cdot)f(\cdot,x_2),\Psi_\ell\ra\Psi_q(x).
\end{equation}
\end{itemize}
For any $\eps>0$, one can choose $N_3(\eps), N_4(\eps), N_6(\eps)$, and $N_5=O(\log(1/\eps))$, $N_7=O(\log(1/\eps))$, such that $K_2(\eta_1,\eta_2)$ can be approximated by the iterative step \eqref{COM1}, \eqref{COM2} with error $\eps$. That is, 
\begin{equation*}
\begin{aligned}
&K_2(\eta_1,\eta_2)(x,x_2)\\
&\approx\sum_{p=1}^{N_5}\sum_{q=1}^{N_7}\left(\sum_{n=1}^{N_3}\sum_{m=1}^{N_4}\sum_{\ell=1}^{N_6} C_{nmp}\widetilde{C}_{\ell q}\la\eta_2,\Psi_n\ra\la\eta_1\Psi_m,\Psi_\ell\ra\right)\Phi_q(x)\Phi_p(x_2),
\end{aligned}
\end{equation*}
which can be represented by the layer structure \eqref{iterative}-\eqref{l-layer}.  

\subsubsection{Neural operators architecture}

Given measurements $\phi(x,x_2)\in L^2(\partial\Omega\times\partial\Omega)$. The target output is $\eta$.

\medskip\noindent
{\bf 1. (Input layer)}. For any $\phi\in Y$, we set 
$$
\hat u_1(y)=\K_1^{N_1,N_2}(\phi):=\sum_{n=1}^{N_1} \left(\sum_{\ell, m=1}^{N_2}C_{\ell m n}^{(0)}\la\phi,\Phi_\ell\otimes\Phi_m\ra_{Y}+b^{(0)}\right)\Psi_n(y).
$$ 

\medskip\noindent{\bf 2. (Hidden layers)}. Choose $M_1\in\N$.

\vspace{2pt}
For $r=2:M_1$

\vspace{3pt}\hspace{5pt} $U_{r-1}=[\hat u_1(y),\hat u_{r-1}(y)]^\top$;

\vspace{3pt}\hspace{5pt} $\hat{u}_r = \K_1^{N_1,N_2}(A^{(r)}\sigma_{L_r}(U_{r-1})),\;\; A^{(r)}\in\R^{1\times d_{L_{M_1}}}$;

\vspace{2pt} End

\medskip\noindent{\bf 3. (Output layer)}. $\eta_{\text{app}}=\sum_{r=1}^{M_1}\hat u_r$.

\medskip To achieve the desired approximation accuracy, the truncation indices $N_1$ and $N_2$ are chosen to satisfy \eqref{NN1}. However, an explicit characterization of the dependence of the hidden-layer rank $J$ in \eqref{KJ} on the accuracy parameter $\varepsilon$ is not available.

\section{Simulation results}\label{sec5}

\subsection{Model complexity for the scattering problem}

The IBS neural operator consists of $M_1 = 3$ Born series terms, corresponding to the number of hidden layers used in the architecture. Each hidden layer is implemented as a 2-layer multilayer perceptron (MLP) with hidden dimension $d = 64$, resulting in the layer structure:

\begin{equation}
\mathbb{R}^8 \rightarrow \mathbb{R}^{64} \rightarrow \mathbb{R}^{64} \rightarrow \mathbb{R}^8
\end{equation}
The input dimension of 8 corresponds to the real-valued representation of the 4-component complex vector $U_{r-1} = [u_1(y), e^{ik\theta\cdot y}, u_{r-1}(z), e^{-ik\hat{x}\cdot z}]^T$.
The encoder $\K_1$ uses $N_1 = 512$ spatial basis functions and $N_2 = 12$ angular basis functions.
The learnable coefficients $C_{lmn}^{(0)} \in \mathbb{R}^{2 \times N_2 \times N_2 \times N_1}$ contribute the majority of the model parameters. The total number of learnable parameters is 157,568, distributed as follows:

\begin{table}[H]
\centering
\caption{Parameter breakdown of the IBS neural operator.}
\begin{tabular}{l c c}
\toprule
\textbf{Component} & \textbf{Shape} & \textbf{Parameters} \\
\midrule
Encoder coefficients $C_{lmn}^{(0)}$ & $(2, 12, 12, 512)$ & 147,456 \\
Encoder bias $b_n^{(0)}$ & $(512)$ & 512 \\
Hidden layer local weights & $3 \times (8\times64 + 64\times64 + 64\times8)$ & $\approx 9,600$ \\
Readout weights $A^{(r)}$ & $3 \times 64$ & 192 \\
\bottomrule
\end{tabular}
\end{table}
The proposed IBS neural operator was trained using $400$ synthetic scattering samples generated from compactly supported Mat\'ern Gaussian random fields with smoothness parameter $\nu = 1.5$ and correlation length $\ell = 0.3$, see \cite{stein1999interpolation} for a detailed description of Mat\'ern Gaussian random fields. The forward scattering data were computed by solving the Lippmann-Schwinger equation via fixed-point iteration. Relative complex Gaussian noise  was added to the far-field measurements to simulate realistic experimental conditions.

The model was trained for $35$ epochs using the Adam optimizer with learning rate $5\times10^{-4}$ and weight decay $10^{-5}$. The reconstruction grid size was $32\times32$, with $12$ incident and observation angles and $12$ angular basis functions. The following metrics are used to evaluate reconstruction quality:

\begin{itemize}
    \item \textbf{Relative $L^2$ Error}. Measures the overall reconstruction accuracy:
    \begin{equation}
    E_{\text{rel}} = \frac{\|\eta - \eta_{\text{app}}\|_2}{\|\eta\|_2}
    \end{equation}
    
    \item \textbf{Peak Signal-to-Noise Ratio (PSNR)}: Measures reconstruction quality in decibels:
    \begin{equation}
    \text{PSNR} = 20\log_{10}\left(\frac{\max|\eta|}{\sqrt{\text{MSE}}}\right)
    \end{equation}
    
    \item \textbf{Structural Similarity Index (SSIM)}: Measures structural similarity:
    \begin{equation}
    \text{SSIM} = \frac{(2\mu_\eta\mu_{\eta_{\text{app}}} + c_1)(2\sigma_{\eta,\eta_{\text{app}}} + c_2)}{(\mu_\eta^2 + \mu_{\eta_{\text{app}}}^2 + c_1)(\sigma_\eta^2 + \sigma_{\eta_{\text{app}}}^2 + c_2)}
    \end{equation}
\end{itemize}
All reported metrics are computed on the test dataset containing $120$ independent samples.

\subsection{Experimental results for the inverse scattering problem}

To comprehensively evaluate the performance and generalization capabilities of the IBS neural operator, we conducted a series of controlled experiments. Each experiment is described in detail below.
\subsubsection{In-distribution generalization}
The model was evaluated on independent realizations of the same Mat\'ern random field model used for training ($\nu = 1.5$, $\ell = 0.3$). This experiment assesses the network's ability to generalize to unseen scatterers chosen from the same distribution. 
The results presented in Table~\ref{tab:in_distribution_results} demonstrate the strong performance of the IBS neural operator for in-distribution generalization. Figure~\ref{fig:in_distribution_reconstructions} shows representative reconstructions for four test samples. 
\begin{table}[H]
\centering
\caption{In-distribution generalization performance of the IBS neural operator.}
\begin{tabular}{l c}
\toprule
\textbf{Metric} & \textbf{Value} \\
\midrule
Relative $L^2$ Error & $0.2891 \pm 0.0487$ \\
PSNR (dB) & $23.90$ \\
SSIM & $0.9920 \pm 0.0021$ \\
Best Validation Loss & $2.49 \times 10^{-4}$ \\
\bottomrule
\end{tabular}
\label{tab:in_distribution_results}
\end{table}
\begin{figure}[H]
\label{fig:in_distribution_reconstructions}
    \centering
    \includegraphics[width=0.7\textwidth]{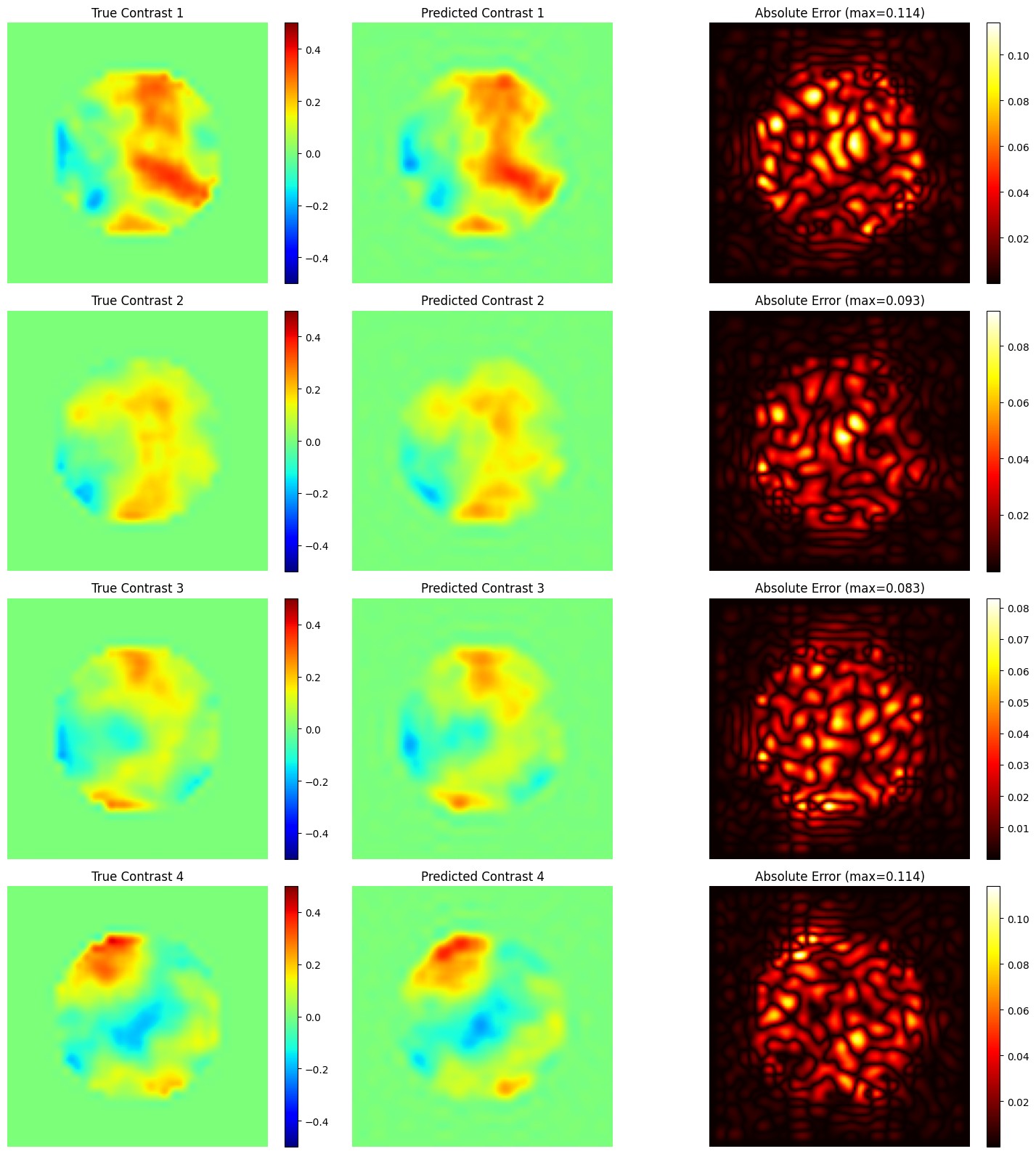}
    \caption{Comparison of true and predicted contrast functions for four representative test samples from the in-distribution test set ($\nu = 1.5$, $\ell = 0.3$)} 
        \end{figure}
\subsubsection{Out-of-distribution generalization: rough fields}

The model was tested on scatterers with rougher texture, characterized by Mat\'ern parameters $\nu = 0.5$ (lower smoothness) and $\ell = 0.1$ (shorter correlation length). This experiment evaluates the network's robustness to unseen roughness patterns.
 \begin{table}[H]
\centering
\caption{Out-of-distribution generalization performance of the IBS neural operator on rough fields ($\nu = 0.5$, $\ell = 0.1$).}
\begin{tabular}{l c}
\toprule
\textbf{Metric} & \textbf{Value} \\
\midrule
Relative  $L^2$ Error & $0.6426 \pm 0.0589$ \\
PSNR (dB) & $18.38$ \\
SSIM & $0.9617 \pm 0.0060$ \\
Best Validation Loss & $2.49 \times 10^{-4}$ \\
\bottomrule
\end{tabular}
\label{tab:ood_a_results}
\end{table}
\begin{figure}[H]
\label{5.1.2}
    \centering
    \includegraphics[width=0.7\textwidth]{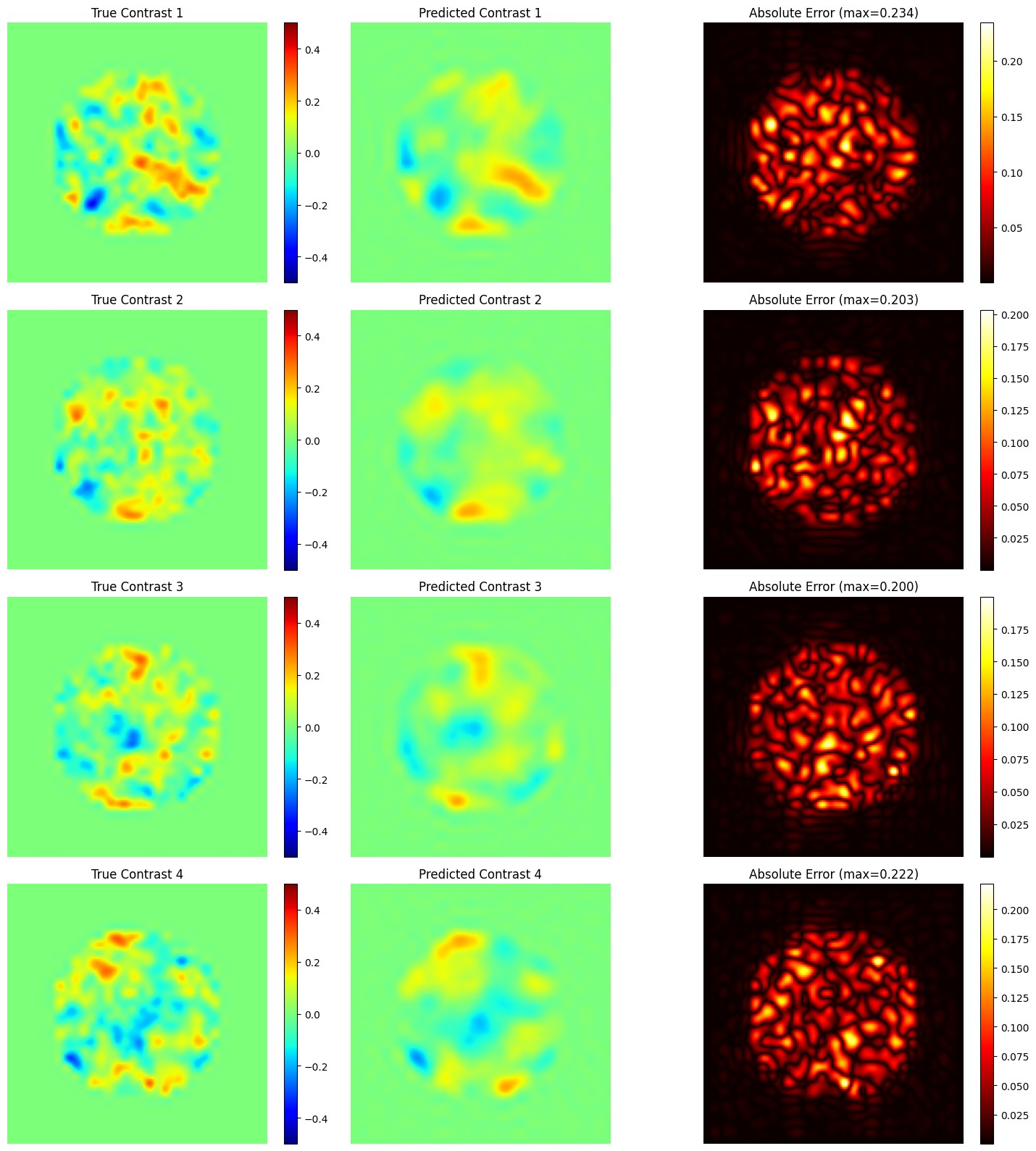}
    \caption{Comparison of true and predicted contrast functions for four representative test samples from the rough field test set ($\nu = 0.5$, $\ell = 0.1$).}
    \end{figure}
   
The IBS neural operator demonstrates robust generalization to rough fields as reported in Table~\ref{tab:ood_a_results}. Despite the significant distribution shift from smooth training data ($\nu = 1.5$, $\ell = 0.3$) to rough test data, the network maintains excellent structural similarity, confirming its ability to capture the location and shape of the scatterers with high accuracy, as demonstrated in Figure~\ref{5.1.2}. The increased relative  $L^2$ error reflects the challenge of recovering higher-frequency components not seen during training, yet the PSNR remains above $18$ dB and SSIM above $0.96$, indicating that the IBS neural operator successfully generalizes to unseen roughness patterns while maintaining reliable structural reconstruction.

\subsubsection{Out-of-distribution generalization: smooth fields}

The model was tested on scatterers with smoother variations, characterized by Mat\'ern parameters $\nu = 2.5$ (higher smoothness) and $\ell = 0.6$ (longer correlation length). This experiment evaluates performance on large-scale, slowly varying features.
\begin{figure}[H]
\label{5.1.3}
    \centering
    \includegraphics[width=0.7\textwidth]{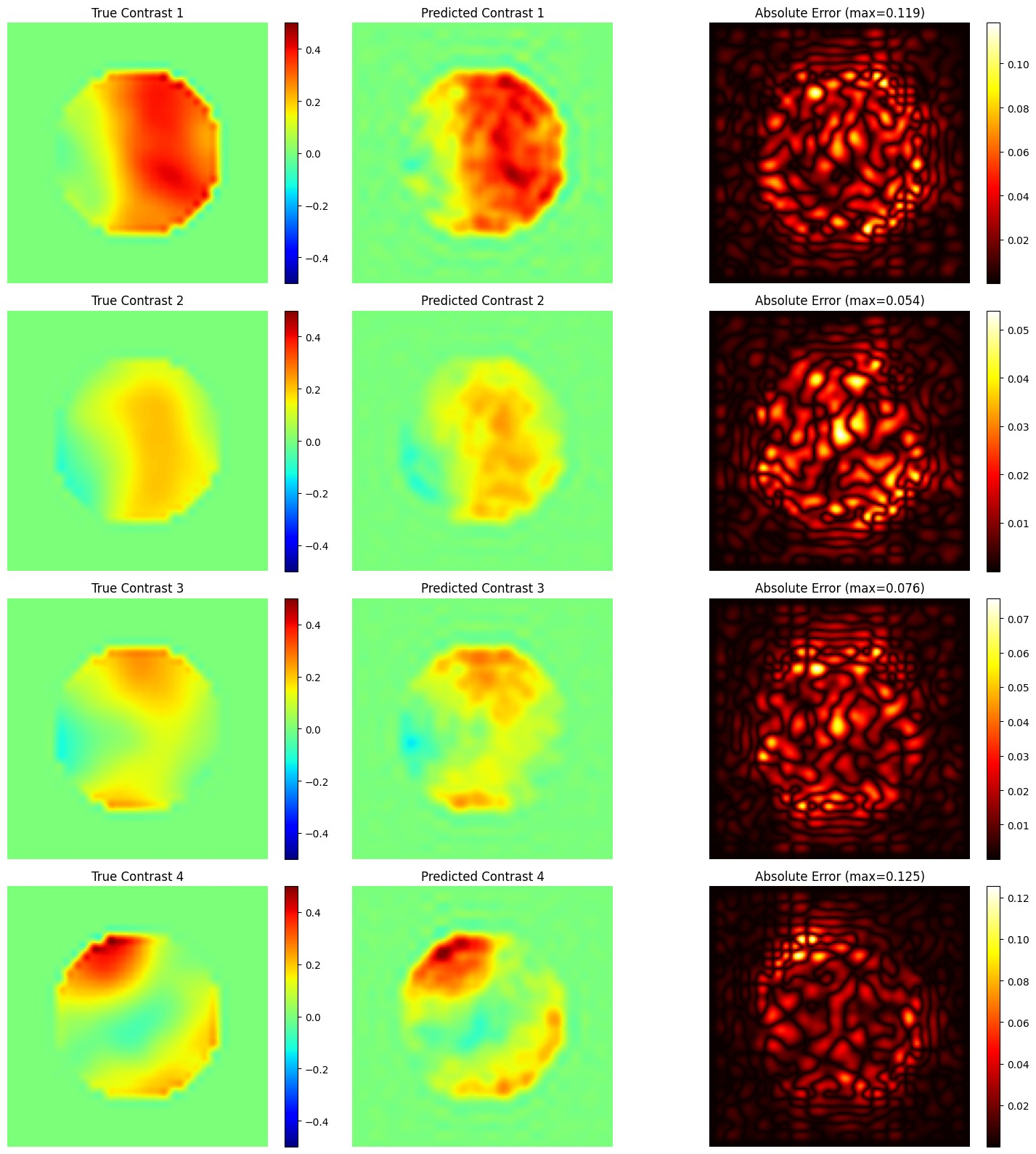}
    \caption{Comparison of true and predicted contrast functions for four representative test samples from the smooth field test set ($\nu = 2.5$, $\ell = 0.6$).}
    \end{figure}
\begin{table}[H]
\centering
\caption{Out-of-distribution generalization performance of the IBS neural operator on smooth fields ($\nu = 2.5$, $\ell = 0.6$).}
\begin{tabular}{l c}
\toprule
\textbf{Metric} & \textbf{Value} \\
\midrule
Relative  $L^2$ Error & $0.2191 \pm 0.0413$ \\
PSNR (dB) & $24.46$ \\
SSIM & $0.9933 \pm 0.0030$ \\
Best Validation Loss & $2.49 \times 10^{-4}$ \\
\bottomrule
\end{tabular}
\label{tab:Smooth_Field}
\end{table}
The IBS neural operator achieves excellent performance on smooth fields as reported in Table~\ref{tab:Smooth_Field}. The network performs even better on smooth fields than on in-distribution data, as demonstrated in Figure~\ref{5.1.3}, most likely because smooth fields are dominated by low-frequency content, making them inherently easier for the network to reconstruct.

\subsubsection{Out-of-distribution generalization: strong scattering}

The model was tested on scatterers with higher contrast, using contrast strength $0.3$ (three times the training value of $0.1$). This experiment evaluates the network's ability to handle strong multiple scattering effects.
\begin{figure}[H]
\centering
\includegraphics[width=0.7\textwidth]{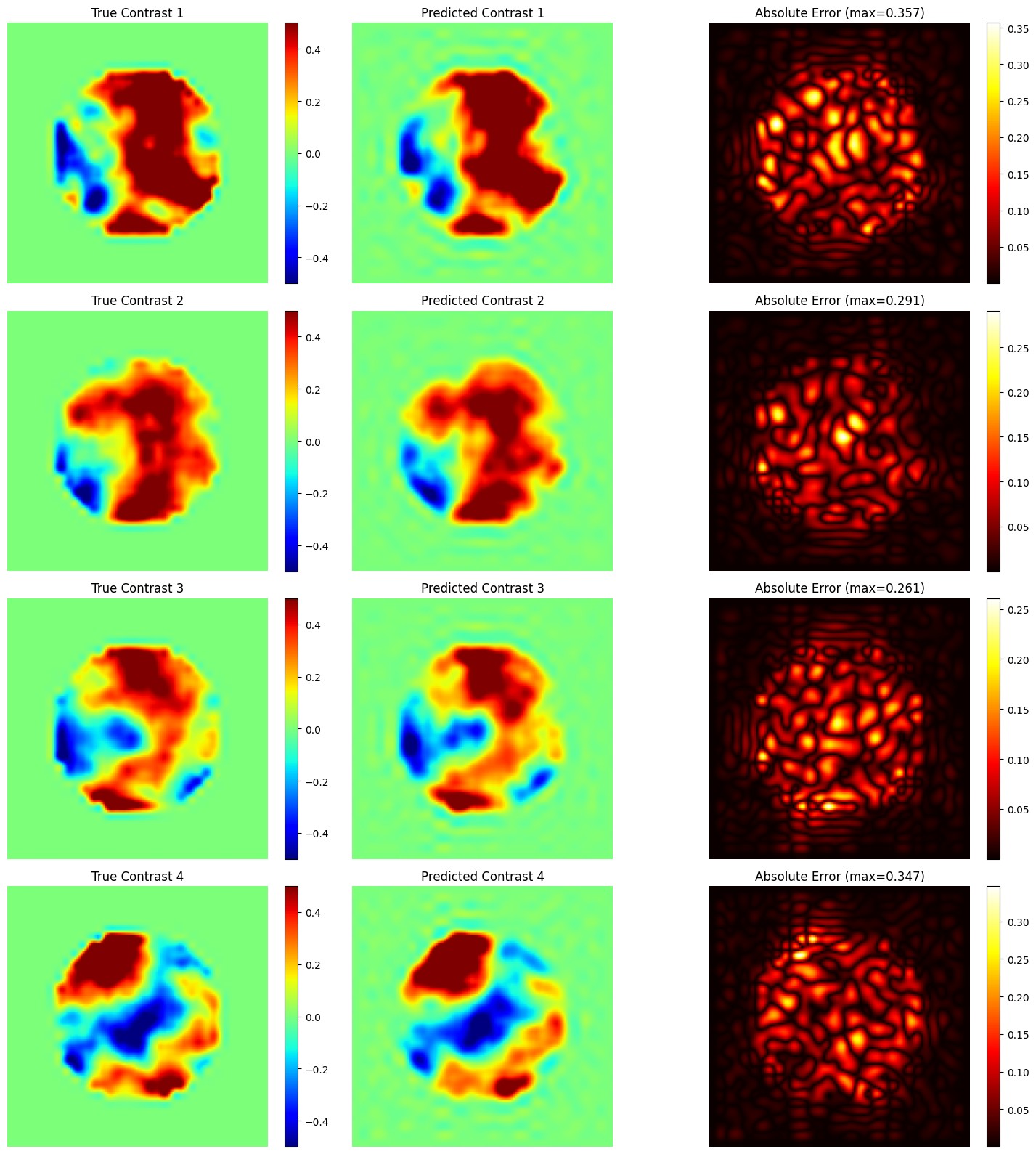}
\caption{Comparison of true and predicted contrast functions for four representative test samples from the strong scattering test set (contrast strength $= 0.3$).}
\label{5.1.4}
\end{figure}
\begin{table}[H]
\centering
\caption{Out-of-distribution generalization performance of the IBS neural operator on strong scattering fields (contrast strength $= 0.3$).}
\begin{tabular}{l c}
\toprule
\textbf{Metric} & \textbf{Value} \\
\midrule
Relative  $L^2$ Error & $0.2825 \pm 0.0460$ \\
PSNR (dB) & $24.09$ \\
SSIM & $0.9714 \pm 0.0068$ \\
Best Validation Loss & $2.49 \times 10^{-4}$ \\
\bottomrule
\end{tabular}
\label{tab:contrast_strength_results}
\end{table}
The IBS neural operator demonstrates strong generalization to strong scattering as reported in Table~\ref{tab:contrast_strength_results}. The SSIM shows a slight decrease compared to in-distribution performance, reflecting the challenge of accurately reconstructing the sharper contrast variations present in strong scattering fields, as demonstrated in Figure~\ref{5.1.4}.

\subsubsection{Robustness to measurement noise}

The model was tested with varying noise levels added to the far-field measurements. The noise was added as complex Gaussian noise scaled by the maximum amplitude of the scattering amplitude:
\begin{equation}
\phi_{\text{noisy}} = \phi_{\text{}} + \epsilon \cdot \max|\phi_{\text{}}| \cdot (\mathcal{N}(0,1) + i\mathcal{N}(0,1))
\end{equation} 
\begin{figure}[H]
\centering
\includegraphics[width=0.7\textwidth]{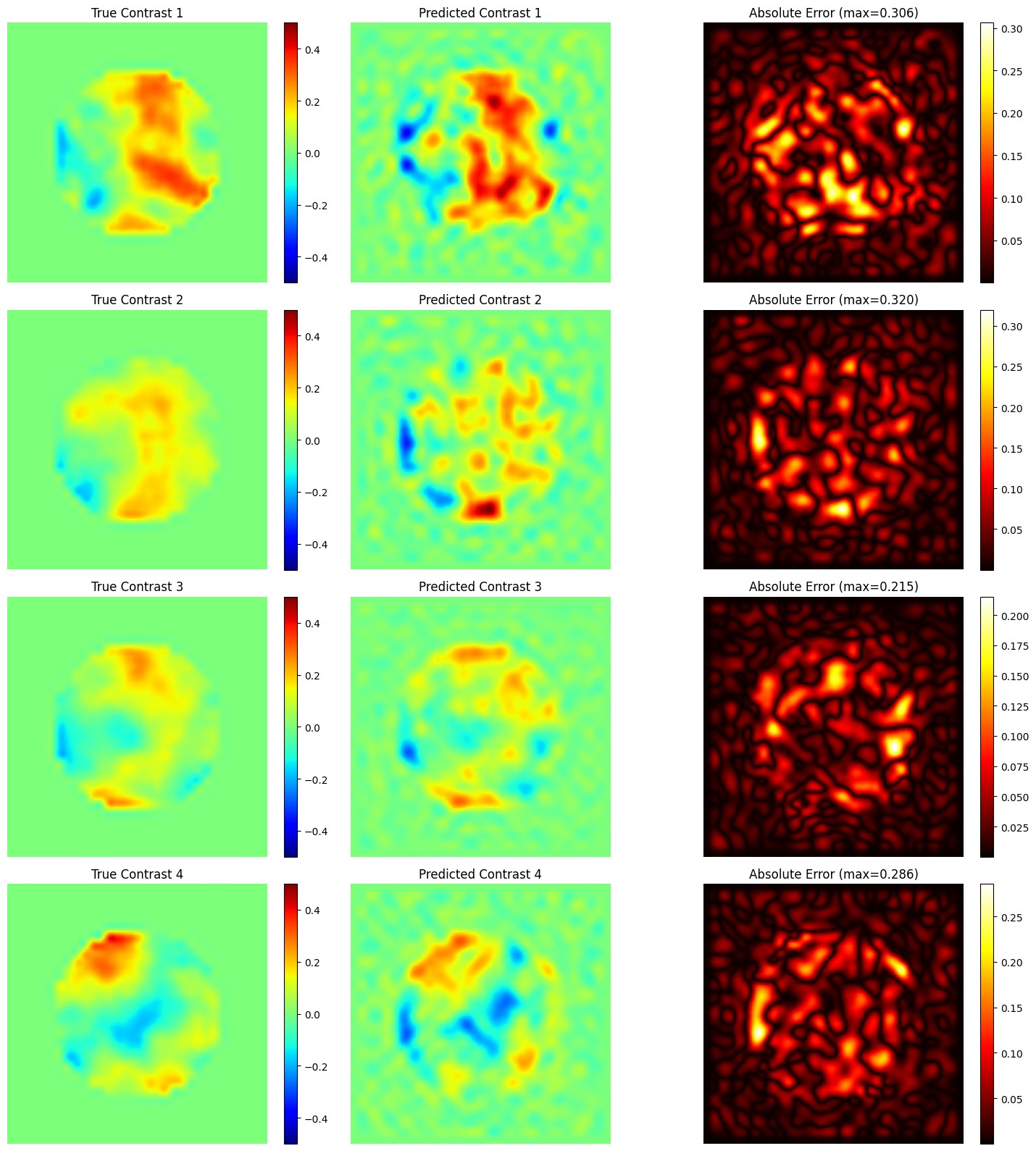}
\caption{Comparison of true and predicted contrast functions for four representative test samples with a high noise level $\epsilon = 0.20$ added to the far-field measurements.}
\label{5.2.5_a}
\end{figure}
\begin{table}[H]
\centering
\caption{Performance of the IBS neural operator on in-distribution test data with a high noise level $\epsilon = 0.20$ added to the far-field measurements.}
\begin{tabular}{l c}
\toprule
\textbf{Metric} & \textbf{Value} \\
\midrule
Relative  $L^2$ Error & $0.7484 \pm 0.1263$ \\
PSNR (dB) & $15.63$ \\
SSIM & $0.9427 \pm 0.0357$ \\
Best Validation Loss & $1.95 \times 10^{-3}$ \\
\bottomrule
\end{tabular}
\label{tab:noise_0.2_results}
\end{table}

At the highest noise level tested ($\epsilon = 0.20$, representing $20\%$ relative noise), the IBS neural operator shows significant decline in reconstruction quality as reported in Table \ref{tab:noise_0.2_results}. Despite this substantial decay, the SSIM remains above $0.94$, indicating that the network still preserves the fundamental structural features of the scatterers even under extreme noise conditions. 
\subsection{Model complexity for the Calder\'on problem}

To investigate whether the neural-operator framework developed for inverse scattering can be extended to the inverse conductivity problem, analogous experiments were performed for the Calderón problem. The same recursive neural-operator framework was adopted, with modifications to accommodate the different forward operator and measurement model arising from the inverse conductivity equation.

The proposed Calderón neural operator consists of an encoder followed by two recursive hidden blocks, producing three successive approximations whose sum yields the reconstructed conductivity perturbation. The encoder $\mathcal{K}_1^{N_1,N_2}$ projects the boundary measurements onto $N_2=32$ boundary basis functions and reconstructs the initial estimate using $N_1=64$ spatial basis functions. To better reflect the geometry of the computational domain, the spatial basis functions $\{\Psi_n\}_{n=1}^{N_1}$ were constructed from the discrete Dirichlet eigenfunctions of the finite element Laplacian on the same polar mesh used for the forward solver.
Each recursive hidden block applies a pointwise two-layer MLP with hidden dimension $d=64$ to the  vector $ U_{r-1}(y)=
\begin{bmatrix}
u_1(y)\\
u_{r-1}(y)
\end{bmatrix},
$
giving the layer structure
$
\mathbb{R}^{2}
\rightarrow
\mathbb{R}^{64}
\rightarrow
\mathbb{R}^{64},
$
followed by a linear readout and a boundary lifting operator that maps the reconstructed field back to the boundary measurement space before reapplying the encoder. The final conductivity perturbation is obtained by summing the three iterative approximations. The complete network contains $2,173,634$ trainable parameters.
The network was trained using $400$ conductivity realizations on a $32\times32$ reconstruction grid with $32$ boundary measurement locations, $32$ boundary basis functions, $64$ spatial basis functions, hidden width $d=64$, batch size $8$, and $35$ training epochs. Optimization was performed using the Adam optimizer with learning rate $3\times10^{-4}$ and weight decay $10^{-4}$.

The training data were generated using the full nonlinear forward model. Specifically, boundary measurements were computed using the finite element method (FEM) forward solver described in \cite{arridge2012inverse}. For each conductivity realization, a Gaussian random field $g$ with Matérn covariance having smoothness parameter $\nu=1.5$ and correlation length $\ell=0.3$ was first generated. The conductivity perturbation was then defined as
$\eta=\exp(g)-1,$ so that $\sigma=\sigma_0(1+\eta)=\sigma_0\exp(g),$
thereby guaranteeing positive conductivity while producing compactly supported perturbations. The conductivity equation \eqref{eq:pde} with Robin boundary conditions was subsequently solved on a polar triangular finite element mesh, and boundary measurements corresponding to dipole current injections were assembled to form the measurement matrix $\phi(x_1,x_2)$. Consequently, the neural operator was trained using measurements generated by the full nonlinear conductivity forward model.

The boundary measurements $\phi$ are real-valued and represented as a single-channel tensor of shape $(1,A,A)$, where $A=32$ denotes the number of boundary measurement locations. The conductivity perturbation $\eta$ is represented on a Cartesian grid of size $32\times32$. Prior to training, both the boundary measurements and the conductivity perturbations were standardized using the mean and standard deviation computed from the training dataset.

\subsection{Experimental results for Calder\'on's problem}

We evaluate the performance and generalizability of the IBS neural operators through a series of controlled experiments, as detailed below.
\subsubsection{In-distribution generalization}
The model was evaluated on independent realizations of the same Matérn conductivity random field used during training $\nu = 1.5$, $\ell = 0.3$. This experiment assesses the ability of the proposed Calderón neural operator to generalize to unseen conductivity perturbations drawn from the same statistical distribution.
The results presented in Table~\ref{tab:calderon_indistribution_results} and figure \ref{Image 1} demonstrate the strong in-distribution reconstruction capability of the proposed Calderón neural operator. 
\begin{table}[H]
\centering
\caption{In-distribution generalization performance of the proposed Calderón neural operator.}
\begin{tabular}{l c}
\toprule
\textbf{Metric} & \textbf{Value} \\
\midrule
Relative  $L^2$ Error & $0.4282 \pm 0.1004$ \\
PSNR (dB) & $24.88$ \\
SSIM & $0.9953 \pm 0.0021$ \\
Best Validation Loss & $1.41 \times 10^{-1}$ \\
\bottomrule
\end{tabular}
\label{tab:calderon_indistribution_results}
\end{table}
\begin{figure}[H]
\centering
\includegraphics[width=0.7\textwidth]{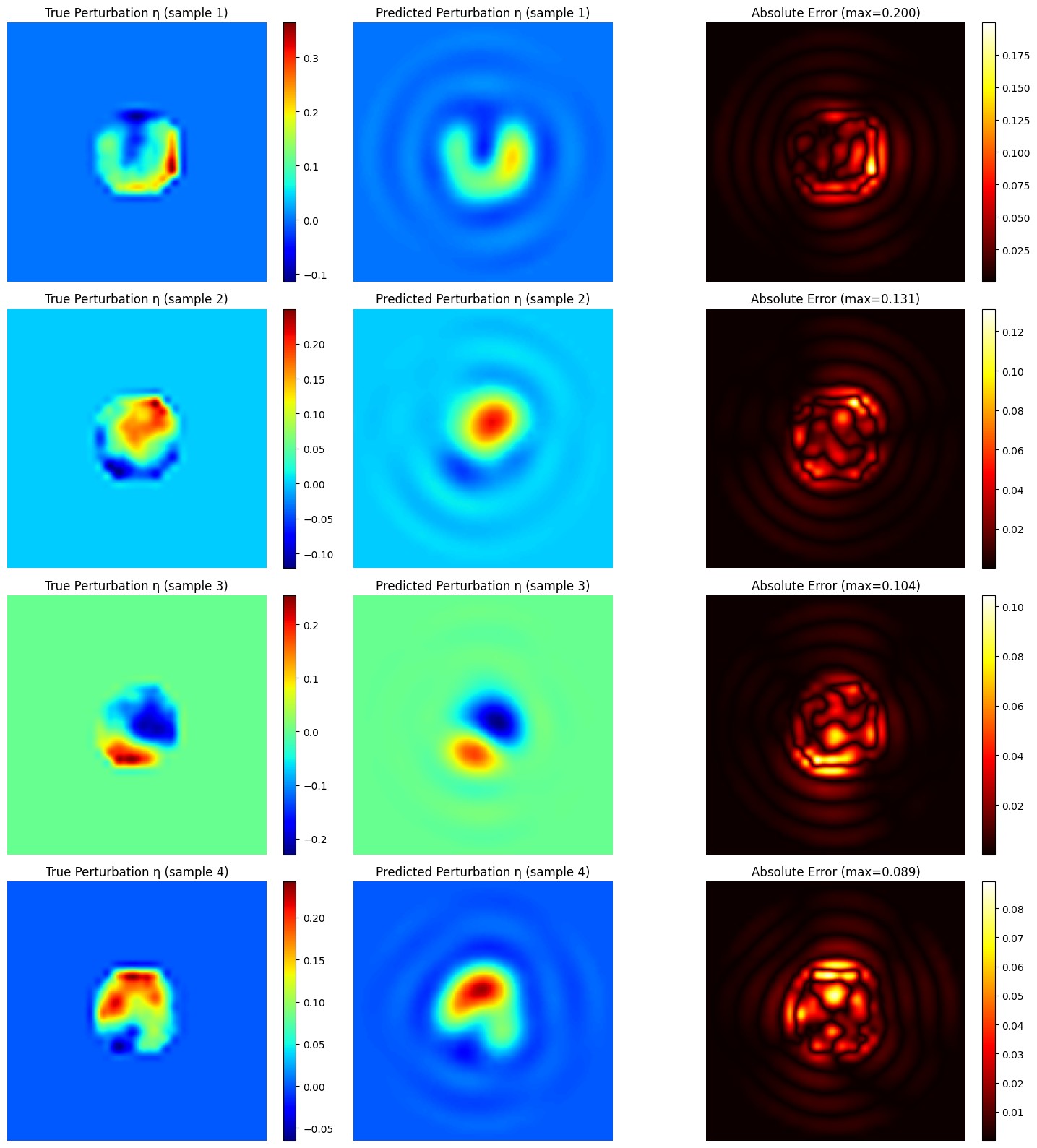}
\caption{Comparison of the true and reconstructed conductivity perturbations for four representative test samples from the in-distribution test set.}
\label{Image 1}
\end{figure}

\subsubsection{Out-of-distribution generalization, different Mat\'ern smoothness}

To evaluate the robustness of the proposed Calder\'on neural operator under a change in distribution, the trained model was tested on conductivity realizations generated from a different Mat\'ern random field. Specifically, while the network was trained using $\nu=1.5$, the test conductivities were generated using $\nu=0.5$. All other simulation parameters, including the correlation length, conductivity strength, and measurement configuration, were kept unchanged. This experiment evaluates the ability of the neural operator to generalize beyond the conductivity prior encountered during training.
The results presented in Table~\ref{tab:calderon_ood_results} and figure \ref{Image 2} demonstrate that the proposed Calderón neural operator remains robust under a moderate distributional shift. As expected, the reconstruction accuracy decreases compared with the in-distribution case, reflecting the increased difficulty of reconstructing conductivity perturbations drawn from a different statistical prior. Although the relative $L^2$ error increases, both the PSNR and SSIM remain high, indicating that the neural operator generalizes well beyond the training distribution.

\begin{table}[H]
\centering
\caption{Out-of-distribution generalization performance of the proposed Calderón neural operator.}
\begin{tabular}{l c}
\toprule
\textbf{Metric} & \textbf{Value} \\
\midrule
Relative  $L^2$ Error & $0.5434 \pm 0.1063$ \\
PSNR (dB) & $23.73$ \\
SSIM & $0.9922 \pm 0.0026$ \\
\bottomrule
\end{tabular}
\label{tab:calderon_ood_results}
\end{table}
\begin{figure}[H]
\centering
\includegraphics[width=0.7\textwidth]{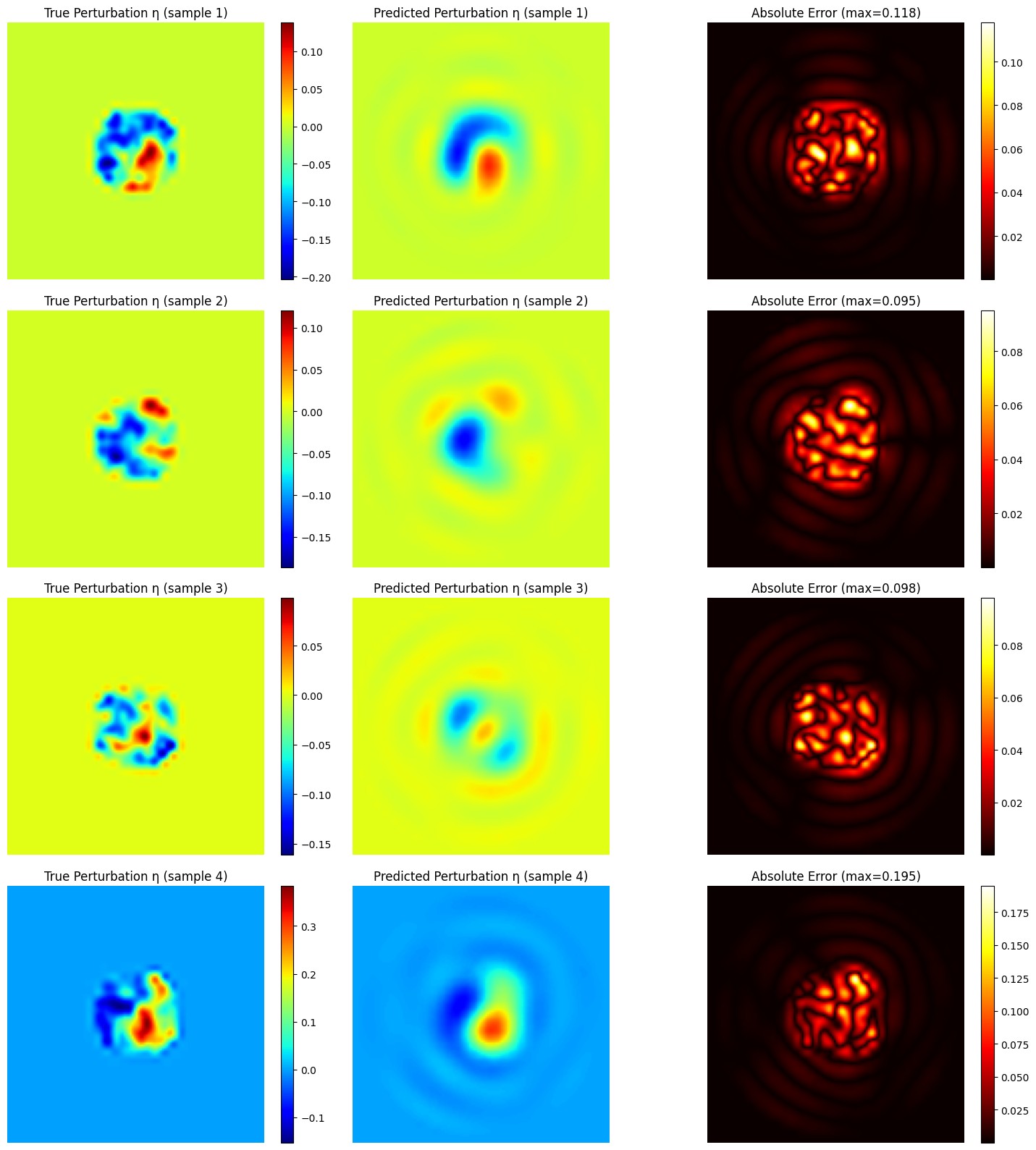}
\caption{Comparison of the true and reconstructed conductivity perturbations for four representative out-of-distribution test samples.}
\label{Image 2}
\end{figure}

\subsubsection{Out-of-distribution generalization, different correlation length}
The model was further evaluated on conductivity perturbations generated using a larger Matérn correlation length while keeping all other parameters fixed. Specifically, the network was trained using conductivity realizations with correlation length $\ell=0.3$, whereas the out-of-distribution test set was generated using $\ell=0.8$. This experiment evaluates the ability of the proposed neural operator to generalize to conductivity fields with significantly smoother spatial variations than those encountered during training.

The results presented in Table~\ref{tab:calderon_ood_length08_results} and Figure~\ref{Image 5} demonstrate that the proposed Calderón neural operator generalizes well to this out-of-distribution setting. Interestingly, unlike the previous out-of-distribution experiments, the reconstruction accuracy does not decline. This behavior is expected since increasing the Matérn correlation length produces smoother conductivity perturbations and the inverse problem becomes less challenging, allowing the network to reconstruct the conductivity perturbations with slightly improved accuracy despite the distributional shift.
\begin{figure}[H]
\centering
\includegraphics[width=0.7\textwidth]{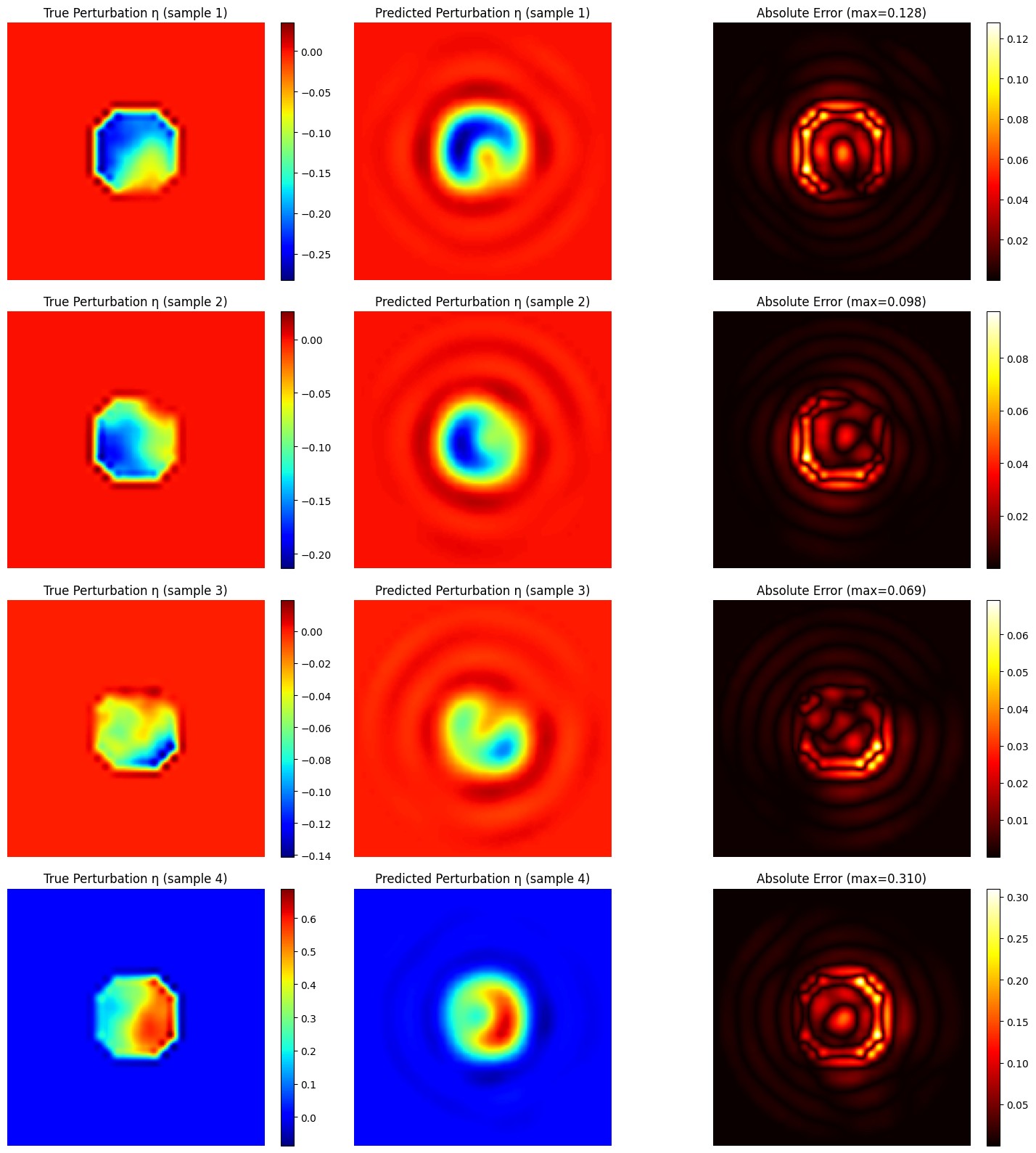}
\caption{Comparison of the true and reconstructed conductivity perturbations for four representative out-of-distribution test samples generated using $\ell=0.8$}
\label{Image 5}
\end{figure}
\begin{table}[H]
\centering
\caption{Out-of-distribution generalization performance of the proposed Calderón neural operator for conductivity perturbations generated using a larger Matérn correlation length.}
\begin{tabular}{l c}
\toprule
\textbf{Metric} & \textbf{Value} \\
\midrule
Relative  $L^2$ Error & $0.3496 \pm 0.0642$ \\
PSNR (dB) & $24.66$ \\
SSIM & $0.9915 \pm 0.0085$ \\
\bottomrule
\end{tabular}
\label{tab:calderon_ood_length08_results}
\end{table}
\subsubsection{Out-of-distribution generalization, stronger conductivities}

In this experiment, the network was trained using conductivity perturbations generated with $g_{\mathrm{strength}}=0.1$ and the test conductivities were generated using $g_{\mathrm{strength}}=0.2$. All other simulation parameters were kept unchanged. This experiment evaluates the ability of the neural operator to generalize to conductivity perturbations with larger amplitudes than those encountered during training.\\
The results presented in Table~\ref{tab:calderon_ood_strength_results_1} and Figure~\ref{Image 4} indicate that the proposed Calderón neural operator remains robust to increased conductivity contrast. Unlike the previous out-of-distribution experiments, only a marginal change in reconstruction performance is observed. These results demonstrate that the proposed neural operator generalizes well to conductivity perturbations with stronger amplitudes than those used during training.
\begin{table}[H]
\centering
\caption{Out-of-distribution generalization performance for stronger conductivity perturbations}
\begin{tabular}{l c}
\toprule
\textbf{Metric} & \textbf{Value} \\
\midrule
Relative  $L^2$ Error & $0.4250 \pm 0.0844$ \\
PSNR (dB) & $25.01$ \\
SSIM & $0.9839 \pm 0.0074$ \\
\bottomrule
\end{tabular}
\label{tab:calderon_ood_strength_results_1}
\end{table}
\begin{figure}[H]
\centering
\includegraphics[width=0.7\textwidth]{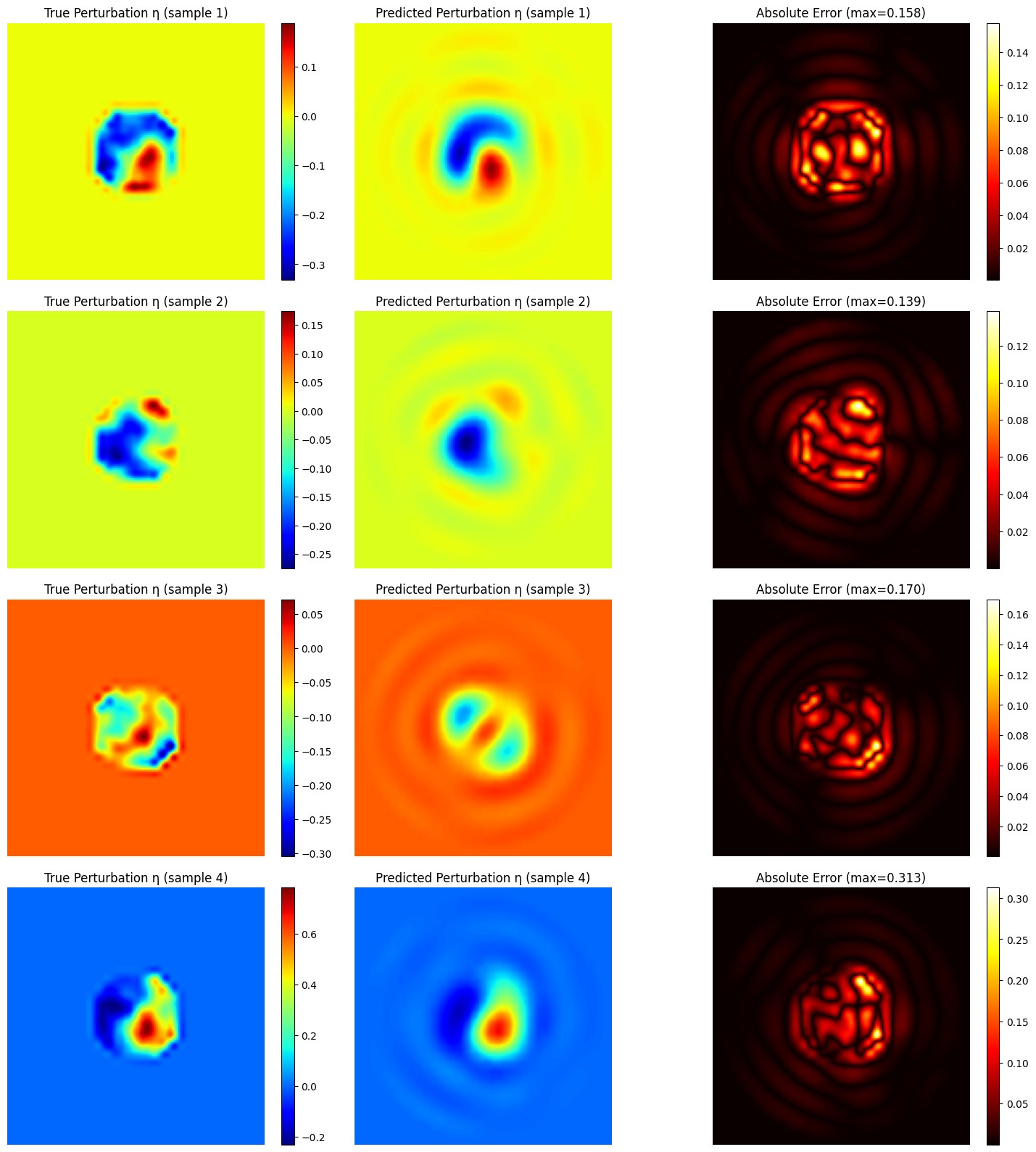}
\caption{Comparison of the true and reconstructed conductivity perturbations for four representative out-of-distribution test samples with stronger conductivity perturbations.}
\label{Image 4}
\end{figure}

\subsubsection{Robustness to measurement noise}

To assess the robustness of the proposed Calder\'on neural operator to measurement errors, additive Gaussian noise was introduced into the simulated boundary measurements. The noisy boundary measurements were generated according to
\begin{equation}
\phi_{\mathrm{noisy}}=
\phi_{\mathrm{}}
+
\epsilon \,
\max\!\left(|\phi_{\mathrm{}}|\right)
\,\xi,
\qquad
\xi \sim \mathcal{N}(0,1),
\end{equation}

where $\phi_{\mathrm{clean}}$ denotes the noise-free boundary measurement matrix, $\epsilon$ is the noise level, and $\mathcal{N}(0,1)$ is a matrix of independent standard Gaussian random variables. In this experiment, a noise level of $\epsilon=0.5$ was considered, corresponding to a severe measurement perturbation.

Figure~\ref{Image 3} illustrates representative reconstructions obtained from noisy boundary measurements. The corresponding quantitative performance is summarized in Table~\ref{tab:calderon_noise_results}.
The results indicate that the proposed Calderón neural operator remains reasonably robust even under severe measurement noise. 
\begin{table}[H]
\centering
\caption{Performance of the proposed Calder\'on neural operator under additive Gaussian measurement noise with $\epsilon=0.5$.}

\medskip
\begin{tabular}{l c}
\toprule
\textbf{Metric} & \textbf{Value} \\
\midrule
Relative  $L^2$ Error & $0.6004 \pm 0.1282$ \\
PSNR (dB) & $21.91$ \\
SSIM & $0.9905 \pm 0.0049$ \\
\bottomrule
\end{tabular}
\label{tab:calderon_noise_results}
\end{table}
\begin{figure}[H]
\centering
\includegraphics[width=0.7\textwidth]{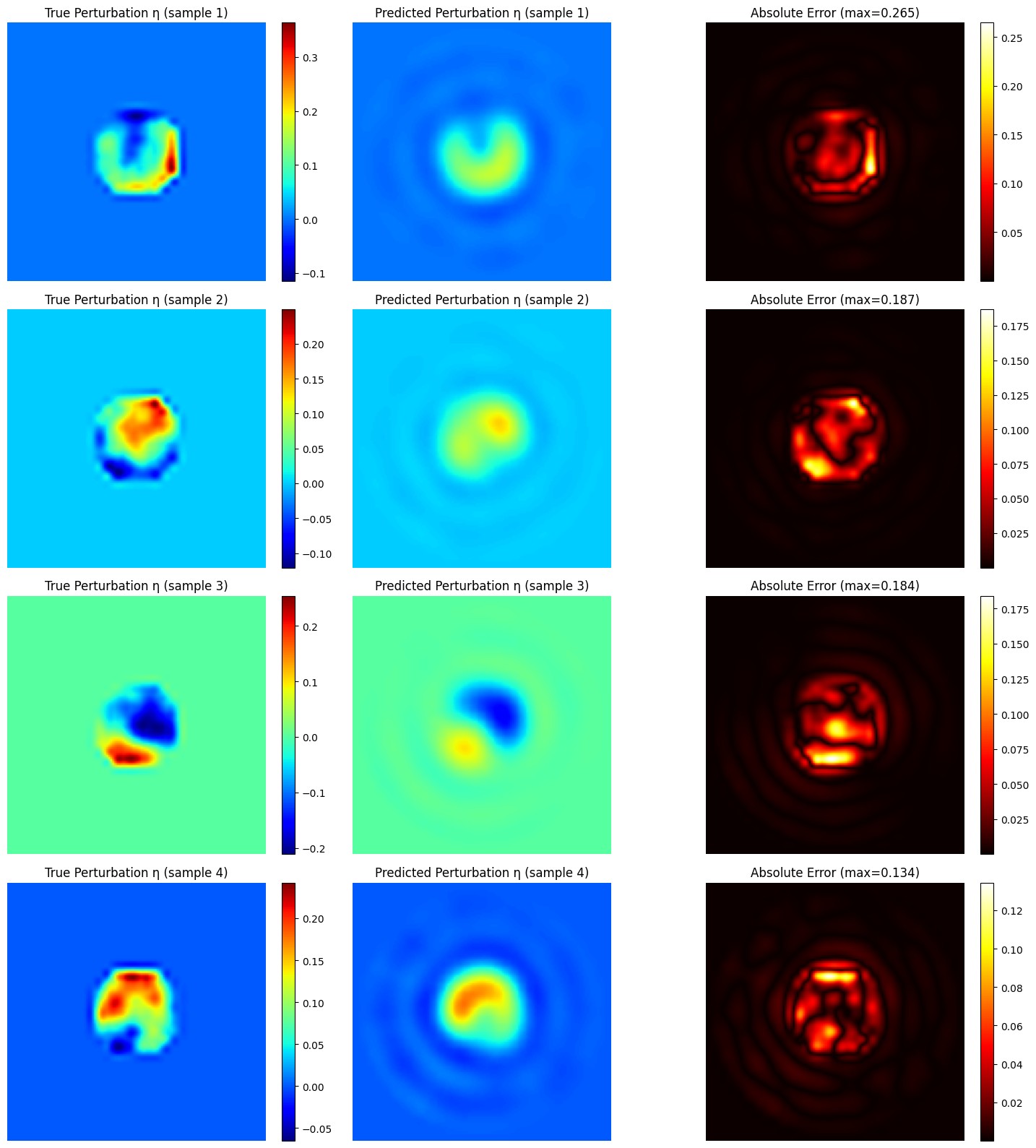}
\caption{Comparison of the true and reconstructed conductivity perturbations for four representative test samples with additive Gaussian noise $\epsilon=0.5$ applied to the boundary measurements.}
\label{Image 3}
\end{figure}
\subsection{Generalization to Discontinuous Conductivity Distributions}
To investigate the limitations of the proposed Calderón neural operator beyond the smooth conductivity distributions encountered during training, two experiments with discontinuous piecewise-constant conductivities were considered.
\subsubsection{Moderate-Contrast Discontinuous Conductivity}
The test samples consisted of randomly positioned and oriented elliptical inclusions with sharp interfaces, whereas the neural operator was trained exclusively on smooth conductivity perturbations generated from Matérn random fields. Synthetic EIT measurements for the discontinuous conductivities were generated using the same nonlinear FEM forward solver, and the trained neural operator was applied. The results in Table~\ref{tab:calderon_discontinuous_moderate} and Figure~\ref{fig:calderon_discontinuous_moderate} show a moderate decline compared with the smooth in-distribution test set.  Nevertheless, the reconstructions retain the approximate locations and overall shapes of the inclusions. The limitation is the recovery of the sharp interfaces, which are reconstructed as smoother transitions, illustrating the difficulty of generalizing  a neural operator trained on smooth conductivity fields to discontinuous targets.

\begin{figure}[H]
\centering
\includegraphics[width=0.7\textwidth]{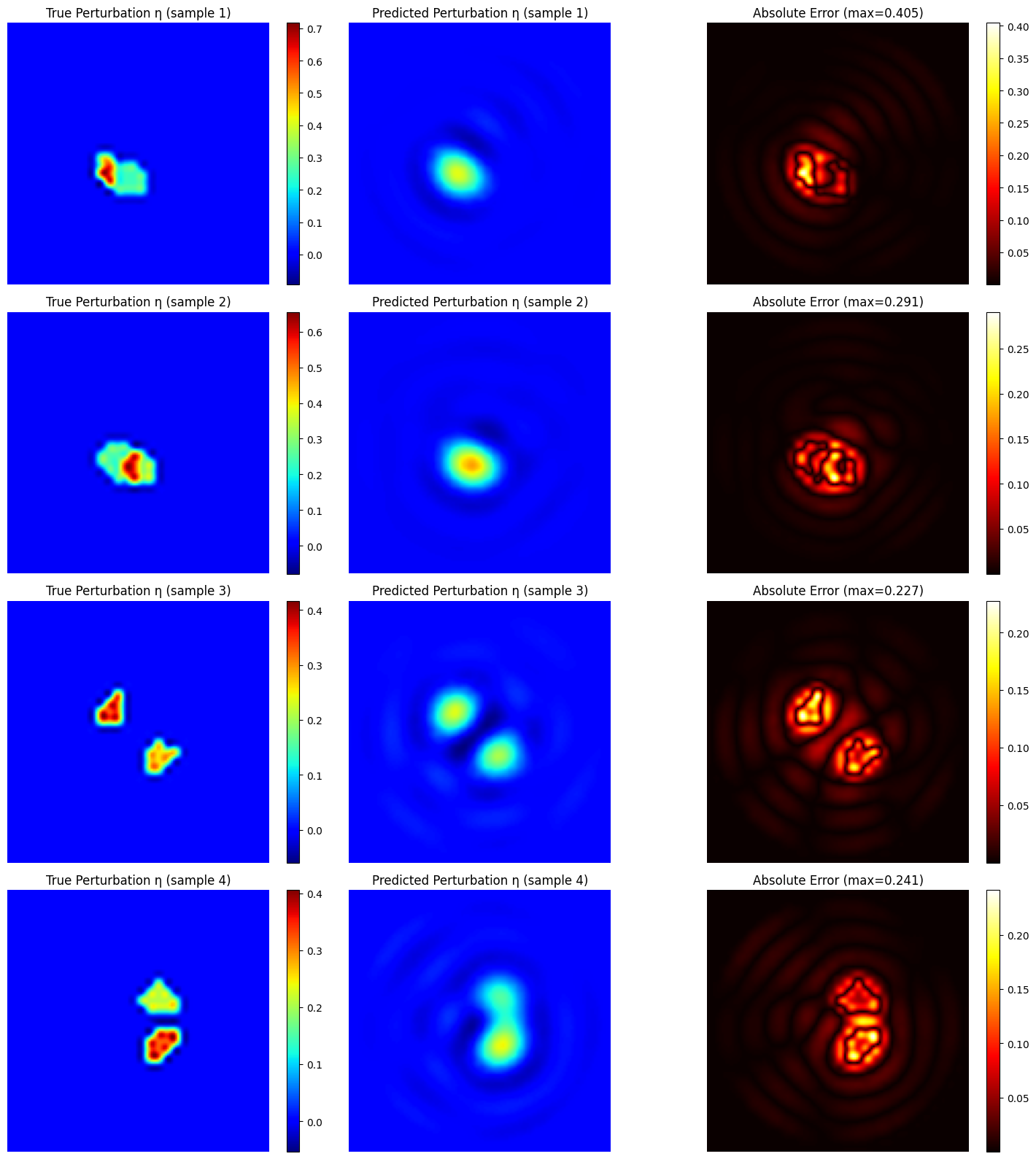}
\caption{Comparison of the true and reconstructed conductivity perturbations for four representative moderate-contrast discontinuous test samples.}
\label{fig:calderon_discontinuous_moderate}
\end{figure}
\begin{table}[H]
\centering
\caption{Reconstruction performance of the proposed Calderón neural operator for moderate-contrast discontinuous conductivity perturbations.}
\begin{tabular}{l c}
\toprule
\textbf{Metric} & \textbf{Value} \\
\midrule
Relative L2 Error & $0.5017 \pm 0.0982$ \\
PSNR (dB) & $24.08$ \\
SSIM & $0.9781 \pm 0.0059$ \\
\bottomrule
\end{tabular}
\label{tab:calderon_discontinuous_moderate}
\end{table}
\subsubsection{High-Contrast Discontinuous Conductivity}
The conductivity was defined to be $\sigma=100$ inside the piecewise-constant inclusions and $\sigma=1$ in the background, corresponding to a conductivity contrast of $100{:}1$.  
\begin{figure}[H]
\centering
\includegraphics[width=0.7\textwidth]{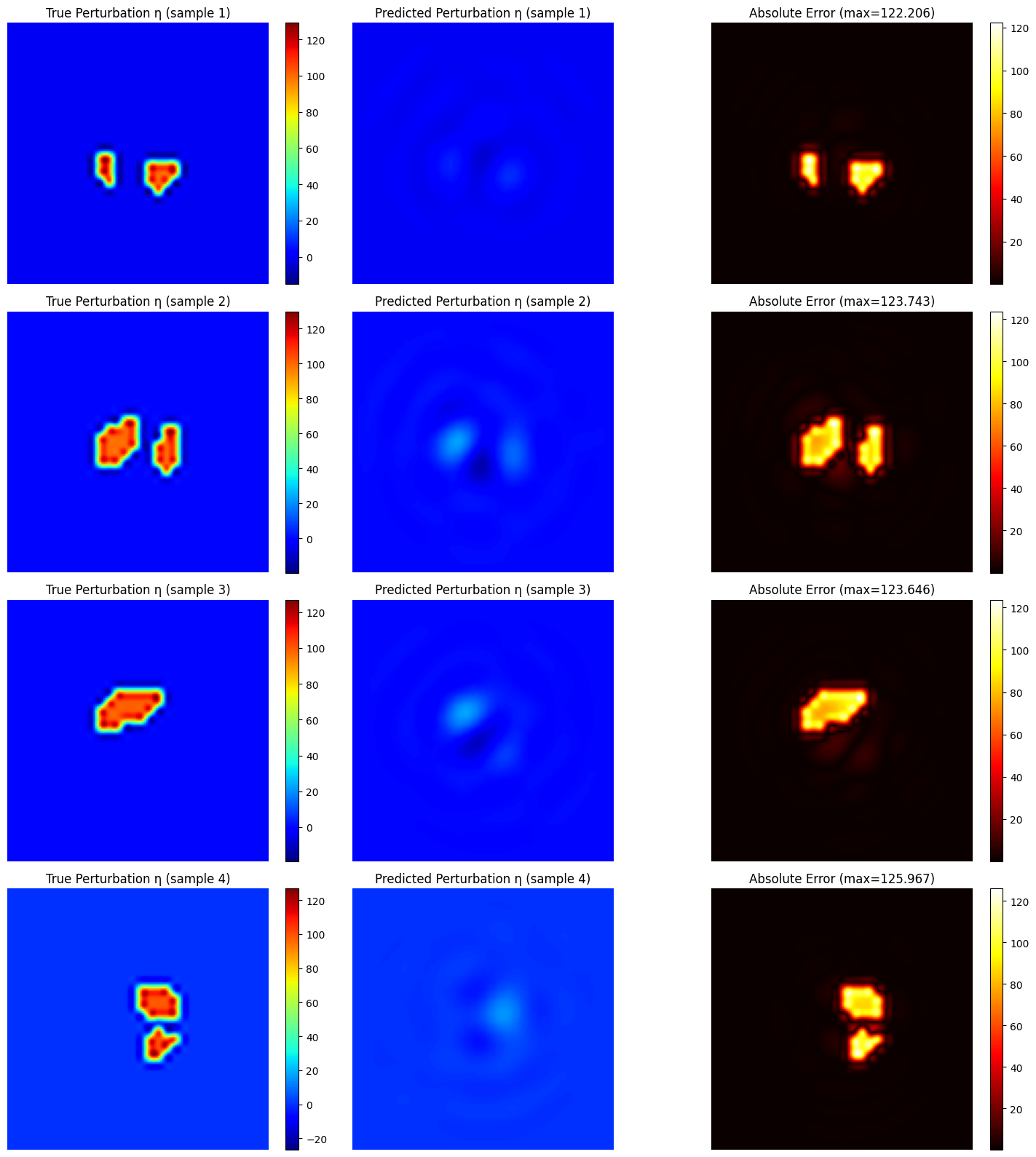}
\caption{Comparison of the true and reconstructed conductivity perturbations for four representative high-contrast discontinuous test samples with $\sigma=100$ inside the inclusions and $\sigma=1$ in the background.}
\label{fig:calderon_discontinuous_high}
\end{figure}
As shown in Table~\ref{tab:calderon_discontinuous_high} and Figure~\ref{fig:calderon_discontinuous_high}, the reconstruction performance declines substantially in this extreme out-of-distribution regime.  Although the reconstructed fields retain some information about the approximate locations of the inclusions, the model fails to recover both their sharp interfaces and their high conductivity amplitudes. This experiment demonstrates a clear limitation of the learned inverse map when extrapolating far beyond the smooth and low-contrast conductivity distribution used during training.

\begin{table}[H]
\centering
\caption{Reconstruction performance of the proposed Calderón neural operator for high-contrast discontinuous conductivities with a conductivity ratio of $100{:}1$.}
\begin{tabular}{l c}
\toprule
\textbf{Metric} & \textbf{Value} \\
\midrule
Relative L2 Error & $0.9205 \pm 0.0224$ \\
PSNR (dB) & $16.40$ \\
SSIM & $0.0228 \pm 0.0108$ \\
\bottomrule
\end{tabular}
\label{tab:calderon_discontinuous_high}
\end{table}

\section{Conclusion}
In this paper, we developed a neural operator framework for nonlinear inverse problems by constructing neural operators that approximate the operator expansions arising from the reduced inverse Born series. The proposed framework was applied to both the inverse scattering and Calder\'on problems. For the scattering problem, the neural operator successfully reconstructed the scattering potentials from synthetic far-field measurements generated by the Lippmann-Schwinger equation. The same architecture was subsequently adapted to the Calderón problem by replacing the forward model and the measurement operator while preserving the overall network construction. Numerical experiments showed that the proposed approach accurately reconstructs both scattering potentials and conductivity perturbations, indicating that it generalizes across different classes of nonlinear inverse problems. Future work includes extending the simulation to three-dimensional problems and establishing a rigorous error analysis for the reduced inverse Born series. Although the reduced inverse Born series provides a natural approximation to the full inverse Born expansion, and both the theoretical construction and numerical results indicate that the neglected higher-order terms introduce only limited error under suitable conditions, a rigorous estimate of the reconstruction error $|\eta-\widehat{\eta}|$ has not yet been established. Developing quantitative error bounds for the reduced series would provide a stronger theoretical foundation for its accuracy and applicability.

\appendix

\section{Finite-rank approximation}

Let $X$ and $Y$ be separable Hilbert spaces with associated orthonormal basis $\{\Phi_n\}_{n\in\N}$ and $\{\Psi_m\}_{m\in\N}$, respectively. Given $T:X\to Y$ a compact operator. For any $M,N\in\N$, let $T_{M,N}:X\to Y$ be defined by
\[
T_{M,N}(\varphi)=\sum_{m=1}^M\sum_{n=1}^N \la T(\Phi_n),\Psi_m\ra_Y\la\varphi,\Phi_n\ra_X.
\]
\begin{theorem}\label{app1}
\[
\lim_{M,N\to\infty}\|T_{M,N}-T\|_{X\to Y}=0.
\]
\end{theorem}
\begin{proof}
Let $P_N$ and $P_M$ be orthogonal projections onto $\{\Phi_n\}_{n=1}^N$ and $\{\Psi_m\}_{m=1}^M$, respectively. Then we can write $T_{M,N}=P_MTP_N$. Now
\[
\|T_{M,N}-T\|_{X\to Y}\le\|T_{M,N}-T_M\|_{X\to Y}+\|T_M-T\|_{X\to Y}=\|T_M(I-P_N)\|_{X\to Y}+\|T_M-T\|_{X\to Y}.
\]
The second term tends to zero as $M\to\infty$, see, for example, \cite[Theorem~4.4]{conway2019course}. That is, for any $\eps>0$, there exists $M_0=M_0(\eps)$ such that $\|T_{M_0}-T\|_{X\to Y}<\eps/2$. For a fixed $M_0$, there exists $N_0=N_0(\eps)$ such that $\|T_{M_0,N_0}-T_{M_0}\|_{X\to Y}<\eps/2$ by the principle of uniform boundedness.  
\end{proof}

\bibliographystyle{plain} 
\bibliography{ref}

\end{document}